\documentclass[11pt]{amsart}

\usepackage{amsmath, amssymb, amsthm, comment}
\usepackage{a4wide}
\usepackage{graphicx}
\usepackage{dsfont}
\usepackage{color}

\theoremstyle{plain}
\newtheorem{tw}{Theorem}[section]
\newtheorem{lem}[tw]{Lemma}
\newtheorem{propn}[tw]{Proposition}
\newtheorem{cor}[tw]{Corollary}

\theoremstyle{definition}

\newtheorem{deft}[tw]{Definition}

\newcommand\Ind{\mathcal{I}}

\newcommand{\br}{\mathbb{R}}

\newcommand{\bn}{\mathbb{N}}
\newcommand{\bc}{\mathbb{C}}
\newcommand{\ltwoH}{L^2(\mlg,\varphi)}

\newcommand{\Dom}{\textup{Dom }}

\newcommand\ltwo{\mathrm{L}^2}

\newcommand{\G}{\mathbb{G}}

\newcommand{\ot}{\otimes}

\newcommand{\id}{\mathrm{id}}

\newcommand{\QG}{\G}
\newcommand{\hQG}{\widehat{\QG}}

\newcommand{\wt}{\widetilde}

\newcommand{\Hil}{\mathsf{H}}

\newcommand{\clg}{\mathsf{C}}
\newcommand{\Four}{\mathcal{F}}

\newcommand{\mlg}{\mathsf{M}}
\newcommand{\cM}{\mathsf{M}}
\newcommand{\cN}{\mathsf{N}}

\newcommand{\nphi}{\mathfrak{n}_{\varphi}}

\newcommand{\mphi}{\mathfrak{m}_{\varphi}}

\newcommand{\Lin}{\tu{Lin} \,}

\newcommand{\la}{\langle}
\newcommand{\ra}{\rangle}

\newcommand{\Tr}{\tu{Tr}}

\newcommand{\hqu}{D_{\varphi}^{\frac{1}{4}}}
\newcommand{\hsq}{D_{\varphi}^{\frac{1}{2}}}

\numberwithin{equation}{section}

\newenvironment{rlist}
{

\begin{enumerate}}
{\end{enumerate}}

\newcommand{\Ww}{\mathds{W}}
\newcommand{\wW}{\text{\reflectbox{$\Ww$}}\:\!}

\newcommand{\tu}{\textup}

\begin{document}

\title[Haagerup property via quantum Markov semigroups]{The  Haagerup approximation property for von Neumann algebras via quantum Markov semigroups and Dirichlet forms}
\begin{abstract}
The Haagerup approximation property for a von Neumann algebra equipped with a faithful normal state $\varphi$ is shown to imply existence of unital, $\varphi$-preserving and KMS-symmetric approximating maps. This is used to obtain a characterisation of the Haagerup approximation property via quantum Markov semigroups (extending the tracial case result due to Jolissaint and Martin) and further via quantum Dirichlet forms.
\end{abstract}

\author{Martijn Caspers}
 \address{M. Caspers, Fachbereich Mathematik und Informatik der Universit\"at M\"unster,
Einsteinstrasse 62,
48149 M\"unster, Germany}
 \email{martijn.caspers@uni-muenster.de}
\thanks{MC is supported by the grant SFB 878 ``{\it Groups, geometry and actions}''}

\author{Adam Skalski}
\address{Institute of Mathematics of the Polish Academy of Sciences,
ul.~\'Sniadeckich 8, 00--956 Warszawa, Poland \newline \indent Faculty of Mathematics, Informatics and Mechanics, University of Warsaw, ul.~Banacha 2,
02-097 Warsaw, Poland}
\thanks{AS is partially supported by the Iuventus Plus grant IP2012 043872.}
\email{a.skalski@impan.pl}


\maketitle


The  Haagerup approximation property for a finite von Neumann algebra $\mlg$ equipped with a faithful normal tracial state $\tau$, motivated by the celebrated Haagerup property for discrete groups (see \cite{book}), was introduced in \cite{Cho}. It asserts the existence of  a family of completely positive, normal, $\tau$-non-increasing  maps on $\mlg$ whose $L^2$-implementations are compact and converge  strongly to the identity.  This property was later studied in depth by P.\,Jolissaint (\cite{Jol}), who showed that it does not depend on the choice of the trace, and that the approximating maps can be chosen unital and trace preserving. If $\Gamma$ is a group with the Haagerup property it follows from the characterisation via a conditionally negative definite function $\psi$ on $\Gamma$ that the corresponding approximating maps for the von Neumann algebra of $\Gamma$ -- which is known to have the Haagerup approximation property -- can be chosen so that they form a semigroup (it is a semigroup of Schur multipliers associated with the positive definite functions arising from $\psi$ by the Sch\"onberg correspondence). In \cite{JM} Jolissaint and F.\,Martin, inspired by \cite{Sau}, showed that this is in fact true in the abstract finite von Neumann algebra setup -- the maps in the definition of the Haagerup approximation property can  always  be chosen so that they form a so-called quantum Markov semigroup on $(\mlg, \tau)$. It is worth noting that the Haagerup property for groups has turned out to play a fundamental role in N.\,Higson and G.Kasparov's approach to Baum-Connes conjecture (\cite{HK}) and the corresponding property for  von Neumann algebras features prominently in S.\,Popa's deformation-rigidity program for II$_1$ factors (\cite{Po2}).

Recent years brought a lot of interest in the Haagerup property for discrete (and locally compact) \emph{quantum groups} (see \cite{DawFimSkaWhi} and references therein). This in turn inspired investigations in the extensions of the Haagerup approximation property for arbitrary von Neumann algebras and two such generalisations were proposed, respectively in \cite{OkaTom} and \cite{CasSka}. Soon it turned out that they are equivalent (we refer to the original preprints and to the note \cite{COST} summarising their content for details of the definitions, see also Section 4 of the current paper). It needs to be noted that the non-tracial context allows far more freedom in choosing a pertinent definition: in particular the relation between the algebra $\mlg$ and the corresponding $L^2$-space becomes subtler. The rich choice of embeddings of $\mlg$ into $L^2(\mlg)$ plays a fundamental role in \cite{OkaTom2}, where the authors also connect the Haagerup property of the von Neumann algebra to the analogous property for the associated $L^p$-spaces.

In the current article we show that in the case of a von Neumann algebra $\mlg$ equipped with a faithful normal state $\varphi$, which has the Haagerup approximation property, one can always choose the approximating maps to be unital and $\varphi$-preserving (in other words \emph{Markov}). This is perhaps surprisingly rather technical, and the most difficult part of the proof is showing that the maps can be chosen to be contractive (Lemma \ref{Lem=Contractive}). Keeping the question of quantum Markov semigroup characterisation in mind, we work in the context of symmetric embeddings, closer in spirit to that of \cite{OkaTom} than to that of \cite{CasSka}. Note that the correspondence between different choices of embeddings was studied in detail in \cite{OkaTom2}, but only under the assumptions that the approximating maps are contractive from the beginning. We would also like to recall that the possibility of choosing approximation built of Markov maps was earlier obtained in \cite{OkaTom} under the additional assumption that $\mlg$ has the \emph{modular Haagerup property} (see Section \ref{Sect=Equiv}).

The fact that the approximating maps can be chosen Markov, and in fact also \emph{KMS-symmetric}, has several applications. First it implies that the Haagerup property of von Neumann algebras with faithful normal states is stable under taking von Neumann algebraic free products, the fact due to F.\,Boca \cite{Boca} in the tracial case and to R.\,Okayasu and R.\,Tomatsu \cite{OkaTom} for the modular Haagerup property. Further it allows us to follow the path (and techniques -- although we promote the Hilbert space point of view) of \cite{JM} and show that  $\mlg$ has the Haagerup approximation property if and only if it admits an \emph{immediately $L^2$-compact KMS-symmetric Markov semigroup}. Here the inspiration comes again from the quantum group world, where \cite{DawFimSkaWhi} shows that the von Neumann algebra of a discrete quantum group with the Haagerup property admits approximating maps forming a semigroup of Schur multipliers (associated to a certain convolution semigroup of states on the dual compact quantum group). We then take one further step and characterise in Theorem \ref{Thm:DirichletHAP} the Haagerup approximation property by the existence of suitable \emph{quantum Dirichlet forms} \`a la S.\,Goldstein and J.M.\,Lindsay (\cite{GL1}) and F. Cipriani (\cite{Cip}). This on one hand opens the perspective towards constructing natural derivations on the von Neumann algebras with the Haagerup approximation property, using the techniques of \cite{CipSau}, and on the other once again connects us to the recent quantum group studies, this time those conducted in by Cipriani, U.\,Franz and A.\,Kula in \cite{FabioUweAnna}, where convolution semigroups of states and related Dirichlet forms were used to construct certain Dirac operators. Finally note that, pursuing the discrete group analogy and remembering that a Dirichlet form may be viewed as an incarnation of the generator of the associated Markov semigroup, we can interpret quantum Dirichlet forms satisfying conditions of Theorem \ref{Thm:DirichletHAP} as abstract von Neumann algebraic counterparts of proper conditionally negative definite functions.

The plan of the article is as follows: in Section 1 we recall the relevant facts concerning crossed products and Haagerup $L^p$-spaces and introduce the necessary notation. Section 2 contains a quick introduction of the symmetric Haagerup property of a von Neumann algebra (soon to be shown to be equivalent to the one studied in \cite{CasSka}), which provides a more convenient framework for the study of quantum Markov semigroups. In Section 3, the most technical part of the paper, we show that the approximating maps for the pair $(\mlg, \varphi)$, where $\mlg$ is a von Neumann algebra with the  Haagerup property and $\varphi$ is a faithful normal state on $\mlg$, can be chosen to be Markov. Section 4 explains the equivalence between various possible notions of the Haagerup property (Markovian or not) for  von Neumann algebras and discusses briefly its modular version. Here we also mention the free product result. Section 5 is devoted to characterising the Haagerup property via existence of immediately $L^2$-compact KMS-symmetric Markov semigroups and Section 6 to characterising the Haagerup property via the existence of quantum Dirichlet forms of particular type.

As in \cite{CasSka} we assume that all the von Neumann algebras considered in the paper have separable preduals.

{\bf Acknowledgement.} We would like to thank R.\,Okayasu and R.\,Tomatsu for useful comments and sharing with us
(after this work was completed) a new, forthcoming version of \cite{OkaTom2}, which  offers an alternative path
to some results in the first part of this article. We also express our gratitude to the anonymous referees for
their careful reading of our article and several useful comments.  We thank Narutaka Ozawa for pointing out a
mistake in an earlier version of this paper.

\section{Preliminaries regarding Haagerup $L^p$-spaces and related embeddings} \label{Sect=HaagerupLp}

We briefly recall basic facts and fix notational conventions regarding the Haagerup non-commutative $L^p$-spaces. For details, see \cite{TerpI}; we will also use some notations of \cite{GL1} and \cite{GL2}. Consider a von Neumann algebra $\cM$ with normal, semifinite, faithful weight $\varphi$. Let $\sigma^\varphi$ be the modular automorphism group of $\varphi$ and let $\cN = \cM \rtimes_{\sigma^\varphi} \mathbb{R}$ be the corresponding crossed product. As usual, $\cN$ is generated by operators $\pi(x), x\in \cM$ and $\lambda_t, t\in \mathbb{R}$. In this section we shall omit $\pi$ and identify $\cM$ with its image in $\cN$.  Let $D_\varphi$ be the generator of the left regular representation, i.e.\ a self-adjoint, positive densely defined operator affiliated to $\cN$ such that  $D_{\varphi}^{it} = \lambda_t, t \in \mathbb{R}$. Let $\tilde{\varphi}$ be the dual weight of $\varphi$, which is a weight on $\cN$, and let $\theta$ be the dual action  of $\mathbb{R}$ on $\cN$.  There exists a unique normal, semifinite faithful trace $\tau$ on $\cN$ such that we have $(D \tilde{\varphi}/D\tau)_t = \lambda_t, t \in \mathbb{R}$ (the cocycle derivative). The Haagerup non-commutative $L^p$-space $L^p(\cM, \varphi)$ (for $p\in [1,\infty]$) is now defined as the collection of all $\tau$-measurable operators $x$ affiliated with $\cN$ such that $\theta_s(x) = e^{-s/p}\: x, s\in \mathbb{R}$. We denote by $L^p(\cM, \varphi)_h$ the self-adjoint part of $L^p(\cM, \varphi)$ and refer to \cite{TerpI} for the norm on $L^p(\cM, \varphi)$. If $x \in \mphi$ and $x = \sum_i y_i^\ast z_i$ (finite sum) with $y_i, z_i \in \nphi$ then
\begin{equation}\label{Eqn=ElementOfLp}
D_\varphi^{\frac{1}{2p}}y_i^\ast \cdot [z_i D_\varphi^{\frac{1}{2p}}] \in L^p(\cM, \varphi).
 \end{equation}
In fact such elements form a dense subset. From this point we will simply write $D_\varphi^{\frac{1}{2p}}x D_\varphi^{\frac{1}{2p}}$ for \eqref{Eqn=ElementOfLp}. The operator does not depend on the decomposition of $x$ as a sum $\sum_i y_i^\ast z_i$ (for a precise construction of the map $ x \mapsto D_\varphi^{\frac{1}{2p}}x D_\varphi^{\frac{1}{2p}}$ for infinite $\varphi$ we refer to Section 2 of \cite{GL1}, where it is denoted by $i^{(p)}$).  Note that when $\varphi$ is a normal, faithful \emph{state} then $D_\varphi^{\frac{1}{p}} \in L^p(\cM, \varphi)$ and we are dealing with the standard product of $\tau$-measurable operators. We will in fact be mainly concerned with the situation $p=2$. In particular $J_{\varphi}$ (or simply $J$) will denote the \emph{modular conjugation} acting on $L^2(\cM,\varphi)$ and $\Lambda_{\varphi}$ (or $\Lambda$) the corresponding \emph{GNS embedding}, so that $\Lambda(x) = x D_{\varphi}^{\frac{1}{2}}$ for $x \in \cM$ such that $\varphi(x^*x)< \infty$.

The following result for the state case is Theorem 5.1 of \cite{HaaJunXu}. It can be extended to the case of a normal semifinite faithful weight without much difficulty, see also Remark 5.6 of \cite{HaaJunXu}.

\begin{lem}\label{Lem=L2Interpolation}
Let $\cM, \cN$ be von Neumann algebras with normal, semifinite, faithful weights $\varphi$, respectively $\psi$. Let $\Phi: \cM \rightarrow \cN$ be a positive map such that $\psi \circ \Phi \leq \varphi$. Then there exists a bounded map:
\[
\Phi^{(2)}: L^2(\cM, \varphi) \rightarrow L^2(\cN, \psi): D_\varphi^{\frac{1}{4}} x D_{\varphi}^{\frac{1}{4}} \mapsto D_\psi^{\frac{1}{4}} \Phi(x) D_\psi^{\frac{1}{4}}.
\]
We call $\Phi^{(2)}$ the KMS $L^2$-implementation of $\Phi$.
\end{lem}

The above statement (suitably reformulated) holds also for other values of $p$. The key steps in the proof are formed by a Radon-Nikodym type result (see for example Proposition 3.2 of \cite{GL2}), the fact that the $L^1$-norm can be expressed via positive elements (Lemma 5.2 of \cite{HaaJunXu}) and interpolation.

\subsection*{In the remainder of the preliminaries assume that $\varphi$ is a normal faithful state}
Given a completely positive normal map $\Phi:\mlg\to \mlg$ which is $\varphi$-reducing (i.e.\ $\varphi \circ \Phi \leq \varphi $)  we can consider its \emph{GNS $L^2$-implementation} $T \in B(\Hil_{\varphi})$, where $\Hil_{\varphi}$ is the GNS Hilbert space for $(\mlg, \varphi)$, given by the formula
\[ T (x\Omega_{\varphi}) = \Phi(x) \Omega_{\varphi}, \;\;\; x \in \mlg,\]
where $\Omega_{\varphi} \in \Hil_{\varphi}$ is the GNS vector of $\varphi$ (note that such implementations were employed in \cite{CasSka}).

The Haagerup space $L^2(\mlg, \varphi)$ is of course naturally isometrically isomorphic to $\Hil_\varphi$, so that in the Haagerup picture the GNS $L^2$-implementation is given by the formula
\[ T( x \hsq ) = \Phi(x) \hsq , \;\;\; x \in \mlg.\]

It is easy to see that if $\Phi$ commutes with the modular automorphism group of $\varphi$ then both $L^2$-implementations coincide.

\begin{deft}
Let $(\mlg,\varphi)$ be a pair of a von Neumann algebra with a faithful normal state. We will say that a map $\Phi:\mlg \to \mlg$ is \emph{Markov}, if it is $\sigma$-weakly continuous (i.e. normal), unital, completely positive and $\varphi$-preserving (note that in fact a variety of versions of this definition appears in literature, often requiring only that for example $\Phi(1)\leq 1$).
\end{deft}

The following result can be read out for example from Section 2 of \cite{GL1} (specifically the discussion before Proposition 2.3 in \cite{GL1} and also Section 4 of \cite{GL2}) -- note that the complete positivity follows from the result for usual positivity, as when we consider the matrix lifting $M_n \ot \mlg$ with the faithful normal state $\textup{tr} \ot \varphi$ we obtain $(\id_n \ot \Phi)^{\dagger} = (\id_n \ot \Phi^{\dagger})$.

\begin{lem}
Let $\Phi:\mlg \to \mlg$ be a Markov map. Then there exists a Markov map $\Phi^{\dagger}:\mlg \to \mlg$ such that for all $x, y \in \mlg$
\begin{equation} \langle \hqu \Phi(x) \hqu, \hqu y \hqu \rangle = \langle \hqu x \hqu, \hqu \Phi^{\dagger}(y) \hqu \rangle. \label{dagger}\end{equation}
\end{lem}

The map $\Phi^{\dagger}$, uniquely determined by \eqref{dagger}, is called the \emph{KMS-adjoint} of $\Phi$. Using the language of the KMS-implementations  the formula $\eqref{dagger}$ means simply, by density of $\hqu \mlg \hqu$ in $L^2(\mlg, \varphi)$, that  $(\Phi^{(2)})^* = (\Phi^{\dagger})^{(2)}$.

\begin{deft}
We say that a Markov map $\Phi:\mlg \to \mlg$ is \emph{KMS-symmetric} if $\Phi=\Phi^{\dagger}$.
\end{deft}

\begin{lem} \label{KMS2M}
If $T\in B(\ltwoH)$ is a self-adjoint, completely positive operator (the latter in the sense of \cite{OkaTom}, note that the representation of $\mlg$ on $\ltwoH$ is canonically a standard form representation) such that for all $\xi \in \ltwoH$ the inequality $0\leq \xi \leq \hsq$ implies $0\leq T\xi \leq \hsq$ and $T(\hsq) = \hsq$ then there exists a KMS-symmetric Markov operator $\Phi$ on $\mlg$ such that $T= \Phi^{(2)}$.
\end{lem}

\begin{proof}
The fact that $0\leq \xi \leq \hsq$  implies $0\leq T\xi \leq \hsq$ guarantees the existence of a positive map $\Phi$ on $\mlg$ such that $T= \Phi^{(2)}$, see the proof of Theorem 4.8 in \cite{OkaTom}. The map $\Phi$ (which is $\Phi_n$ in \cite{OkaTom}) is constructed in the first paragraph of the proof of that theorem via Lemma 4.1 of \cite{OkaTom}. Recall that since $T$ is self-adjoint, $\Phi$ is KMS-symmetric. The complete positivity of $T$ allows us to transport this argument to all matrix levels.
\end{proof}

The operators on $\ltwoH$ satisfying the assumptions of the above lemma will be called \emph{KMS-symmetric $L^2$-Markov operators}.

Finally note that for Markov maps $(\Phi_n)_{n=1}^{\infty}$ and another Markov map $\Phi$ the convergence $\Phi_n \stackrel{n\to\infty}{\longrightarrow}\Phi$ point-$\sigma$-weakly coincides with the weak convergence $\Phi_n^{(2)} \stackrel{n\to\infty}{\longrightarrow}\Phi^{(2)}$ -- this follows as then the maps $\Phi_n^{(2)}$ and $\Phi^{(2)}$ are contractions. One can also achieve strong convergence of the $L^2$-implementations, passing, if necessary, to convex combinations of the original maps (see Theorem II 2.6 (IV) in \cite{TakI}). Further using once more contractivity of the maps in question we see that for Markov maps the point-$\sigma$-weak convergence $\Phi_n \stackrel{n\to\infty}{\longrightarrow}\Phi$  is equivalent to the following:
\[ \langle \hqu \Phi_n(x) \hqu, \hqu y \hqu \rangle  \stackrel{n\to\infty}{\longrightarrow} \langle \hqu \Phi(x) \hqu, \hqu y \hqu \rangle, \;\;\; x, y\in \mlg. \]

\section{Definition of the symmetric Haagerup property}\label{Section=SymmetricHAP}

The aim of this section is a short introduction of the symmetric Haagerup property. We will see later in the paper that it is equivalent to the Haagerup property of \cite{COST} (see also \cite{OkaTom2}). We refer to Section \ref{Sect=HaagerupLp}  for conventions on the Haagerup non-commutative $L^2$-space $L^2(\cM, \varphi)$.

\begin{deft}\label{Dfn=HAPSymmetric}
A pair $(\cM, \varphi)$ of a von Neumann algebra $\cM$ with normal, semifinite, faithful weight $\varphi$ is said to have the {\it symmetric Haagerup property} if there exists a sequence $\Phi_k: \cM \rightarrow \cM$ of normal, completely positive maps such that $\varphi \circ \Phi_k \leq \varphi$ and their KMS $L^2$-implementations $\Phi^{(2)}_k \in B(\ltwoH)$  are compact operators converging to the identity of $L^2(\cM, \varphi)$ strongly.
\end{deft}

The difference with the Haagerup property studied in \cite{CasSka} is that Definition \ref{Dfn=HAPSymmetric} uses the {\it symmetric} injection of non-commutative $L^2$-spaces \cite{Kos} instead of the {\it right} injection (KMS embeddings instead of the GNS embeddings).  The usage of KMS embeddings is more suitable for the applications in Section \ref{Sect=Dirichlet}. It is clear that the two possibilities coincide if $\varphi$ is a tracial weight; in fact the two approaches turn out to be equivalent, as follows from results in this paper and \cite{CasSka} (see Section \ref{Sect=Equiv}). It is important to note that in \cite{OkaTom} also the symmetric correspondence between the $L^2$- an $L^\infty$-level was used to establish a suitable notion of the Haagerup approximation property in terms of the standard form of a von Neumann algebra. However, the approach of \cite{OkaTom} starts from completely positive maps (with respect to a positive cone) acting on the standard form Hilbert space $L^2(\cM, \varphi)$ and reconstructs the maps $\Phi_k$ on the von Neumann algebra level from it. It is not clear if the maps $\Phi_k$ behave well with respect to the weight $\varphi$, that is if $\varphi \circ \Phi_k \leq \varphi$ and if the maps $\Phi_k$ can always be chosen contractive. These properties turn out to be crucial in Section \ref{Sect=Dirichlet} and justify a short redevelopment of the theory of \cite{CasSka} for the symmetric Haagerup property.


\section{Markov property of the approximating maps and weight independence}

In this section we prove that the symmetric Haagerup property is in fact equivalent to the `contractive' symmetric Haagerup property,
and deduce from this fact that in the case of a state the approximating maps may be chosen as Markov. Along the same line we prove that
the Haagerup property is independent of the choice of the weight. Let us first recall the following theorem describing the situation in
the tracial case, due to Jolissaint \cite{Jol}.

\begin{tw}[\cite{Jol}]\label{Thm=TracialCase}
Let $\cM$ be a finite von Neumann algebra with normal faithful tracial state $\tau$. If $(\cM, \tau)$ has the Haagerup property, then the completely positive maps $\Phi_k$ of Definition \ref{Dfn=HAPSymmetric} may be chosen unital and trace preserving.
\end{tw}

We will soon need a following simple lemma.

\begin{lem}\label{Lem=ExtraJolissaint}
Let $\cM$ be a semifinite von Neumann algebra with normal semifinite faithful trace $\tau$. If $(\cM, \tau)$ has
the (symmetric) Haagerup property then the approximating cp maps $\Phi_k$ of Definition \ref{Dfn=HAPSymmetric}
may be chosen contractive.
\end{lem}
\begin{proof}
Let $\{ e_n \}_{n \in \mathbb{N}}$ be a sequence of $\tau$-finite projections in $\cM$ converging strongly to 1. Let $\Phi_k$ be the approximating cp maps witnessing the Haagerup property of $(\cM, \tau)$. Then $e_n \Phi_k(\: \cdot \:) e_n$ are cp maps witnessing the Haagerup property of $(e_n\cM e_n, e_n \tau e_n)$. Theorem \ref{Thm=TracialCase} shows that we may replace these maps by contractive maps $\Phi'_k$ witnessing the Haagerup property of $(e_n\cM e_n, e_n \tau e_n)$. Then for suitable $k,n$  the maps $\Phi_{k,n}'(e_n \: \cdot \: e_n)$ will form a net of contractive cp maps witnessing the Haagerup property of $(\cM, \tau)$.
\end{proof}

Quite importantly, Theorem \ref{Thm=TracialCase} implies in particular that one can achieve a uniform bound for operators  $\Vert \Phi_k \Vert = \Phi_k(1)$ (specifically these operators may be assumed to be all dominated by $1$).  We can now prove the following lemma which is the key, most technical step in the proof of Theorem \ref{Thm=MainTheorem}.

\begin{lem}\label{Lem=MainResult}
Let $\cM$ be a finite von Neumann algebra equipped with normal faithful tracial state $\tau$. Let $h \in \cM^+$ be boundedly invertible and let $\varphi(\: \cdot \:) =  \tau(h\: \cdot \: h)$. If $(\cM, \varphi)$ has the symmetric Haagerup property, then the completely positive maps $\Phi_k, k \in \mathbb{N}$ in Definition \ref{Dfn=HAPSymmetric} may be chosen contractive.
\end{lem}
\begin{proof} We split the proof into two steps.

\vspace{0.3cm}


\noindent {\bf Step 1:} Firstly, suppose that $\Phi_k$ are the completely positive maps witnessing the symmetric Haagerup property of $(\cM, \varphi)$. Then,
\[
\Psi_k( \: \cdot \:) = h \Phi_k(h^{-1} \: \cdot \: h^{-1} ) h,
\]
witnesses the symmetric Haagerup property for the pair $(\cM, \tau)$ (we leave the details to the reader). Then since we know that $(\cM, \tau)$ has the symmetric Haagerup property, Theorem \ref{Thm=TracialCase} implies that the approximating maps, say this time $\Psi_k'$, may be taken unital and $\tau$-preserving. This implies that,
\[
\Phi_k'(\: \cdot \:) = h^{-1} \Psi_k'(h \: \cdot \: h) h^{-1},
\]
witnesses the Haagerup property for $(\cM, \varphi)$ (again the details are for the reader). The important conclusion is that this means that we may assume that there is a uniform bound in $k$ for the completely positive maps $\Phi_k'$ that witness the symmetric Haagerup property for $(\cM, \varphi)$. Note that indeed an explicit bound is given by $(\Vert h^{-1} \Vert \Vert h \Vert)^2$.

\vspace{0.3cm}

\noindent {\bf Step 2:} Let again $\Phi_k$ be the approximating cp maps witnessing the Haagerup
property of $(\cM, \varphi)$ and by Step 1 assume that $\Vert \Phi_k \Vert$ is uniformly bounded
in $k$. We shall now prove the lemma, that is that   we can take completely positive maps with
$\Phi_k(1) \leq 1$. In order to do so set,
\[
\Phi_k^l(x) = \sqrt{\frac{1}{l \pi  }} \int_{-\infty}^\infty e^{- t^2/l} \sigma_t^{\varphi}(\Phi_k( \sigma_{-t}^\varphi(x  ) )) dt.
\]
Since $\varphi \circ \sigma_t^\varphi = \varphi$ we see that $\varphi \circ \Phi_k^l \leq \varphi$ and its KMS $L^2$-implementation is given by
\begin{equation}\label{Eqn=Tkl}
T_k^l := \sqrt{\frac{1}{l \pi}} \int_{-\infty}^\infty e^{-t^2/l}  h^{it} J h^{it} J \: T_k \: h^{-it} J h^{-it} J \: dt.
\end{equation}
Because $h$ is boundedly invertible the mapping $t \mapsto h^{it}$ is norm continuous as follows from the continuous functional calculus. This implies that \eqref{Eqn=Tkl} is in fact a Bochner integral and hence $T_k^l$ is compact.   Moreover, approximating the integral in norm with  suitable step functions it follows from a $3\epsilon$-argument that for every $l \in \mathbb{N}$ we have $T_k^l \rightarrow 1$ strongly.  Define,
\[
g_k^l := \Phi_k^l(1) = \sqrt{\frac{1}{l \pi}} \int_{-\infty}^\infty e^{-t^2/l} \sigma^\varphi_{t}(\Phi_k(1)) \:dt.
\]
Then $g_k^l$ is analytic for $\sigma^\varphi$ and we have,
\[
\sigma_{-i/4}^\varphi(g_k^l) = \sqrt{ \frac{1}{l \pi} } \int_{-\infty}^{\infty} e^{-(t+i/4)^2/l} \sigma_t^\varphi(\Phi_k(1))\: dt.
\]
We claim that we have,
\begin{equation}\label{Eqn=ClaimNormAnalytic}
\Vert \sigma_{-i/4}^\varphi (g_k^l) - g_k^l \Vert \rightarrow 0 \qquad \textrm{ as } \qquad  l \rightarrow \infty.
\end{equation}
Moreover, this convergence is uniform in $k$. Indeed, let $\epsilon >0$ be arbitrary. Choose $\delta> 0$ such that for every $t \in (-\delta, \delta)$ we have $\vert 1 - e^{it/2} \vert < \epsilon$. Since,
\[
2 \sqrt{\frac{1}{l \pi}} \int_{\delta l}^\infty e^{-t^2/l} dt = \frac{2}{\sqrt{\pi}} \int_{\delta \sqrt{l}}^\infty e^{-s^2} ds \rightarrow 0,
\]
as $l \rightarrow \infty$ we may choose $l \in \mathbb{N}$ such that
\[
2 \sqrt{\frac{1}{l \pi}} \int_{\delta l}^\infty e^{-t^2/l} dt < \epsilon \quad \textrm{ and } \quad \vert e^{1/16l} - 1 \vert < \epsilon.
\]
Then we find the following estimates,
\[
\begin{split}
 \Vert \sigma_{-i/4}^\varphi(g_k^l) - g_k^l \Vert \leq & \sqrt{\frac{1}{l \pi}} \int_{-\infty}^\infty \vert  e^{-(t+i/4)^2/l } - e^{-t^2/l} \vert \: dt\: \Vert \Phi_k(1) \Vert \\
\leq & \sqrt{\frac{1}{l \pi}} \int_{-\delta l}^{\delta l} e^{-t^2/l} \vert e^{-(it/2-1/16)/l } -1 \vert \: dt \:\Vert \Phi_k(1) \Vert  \\
& + \:\:  \sqrt{\frac{1}{l \pi}} \int_{(-\infty, -\delta l) \cup (\delta l, \infty)} \!\!\!\!\!\!\!\!\!\!\!\!\!\!\! e^{-t^2/l} \vert e^{-(it/2-1/16)/l } -1 \vert \: dt \:\Vert \Phi_k(1) \Vert.
\end{split}
\]
By our assumptions, the integral on the last line is smaller than $\epsilon (2+ \epsilon) \Vert \Phi_k(1) \Vert$. The integral before the last line is smaller than $\Vert \Phi_k(1) \Vert$ times the maximum over $t \in [-\delta l, \delta l]$ of $\vert e^{-(it/2-1/16)/l} - 1 \vert \leq \vert e^{-it/2l} (e^{1/16l} -1) \vert + \vert e^{-it/2l} -1 \vert$, which by assumption can be estimated from above  by $2 \epsilon \Vert \Phi_k(1) \Vert$.  This proves that \eqref{Eqn=ClaimNormAnalytic} holds. Note that the convergence is uniform in $k$ because we assumed that the maps $\Phi_k$ were uniformly bounded. With some trivial modifications of the argument above we can also prove the following convergence result, which holds uniformly in $k \in \mathbb{N}$:
\begin{equation}\label{Eqn=ClaimNormAnalyticII}
\Vert \sigma_{-i/2}^\varphi (g_k^l) - g_k^l \Vert \rightarrow 0 \qquad \textrm{ as } \qquad  l \rightarrow \infty.
\end{equation}

Next, let $F_n (z) = e^{-n (z-1)^2}$ with $z\in \mathbb{C}$. In particular $F_n$ is holomorphic. Set $f_k^{n,l} = F_n(g_k^l)$. From holomorphic functional calculus applied to the Banach algebra of $\cM$-valued functions on the strip $\mathcal{S} = \{ z \in \mathbb{C} \mid -1/2 \leq {\rm Im}(z) \leq 0 \}$ that are analytic on the interior and norm-continuous on its boundaries we see that $z \mapsto \sigma_z^{\varphi}(f_k^{n,l}) ( = F_n(z \mapsto \sigma_z^\varphi(g_k^l)))$ must extend to a function that is analytic on the interior of $\mathcal{S}$ and continuous on its boundaries.

We are now ready to prove that there exists a sequence of completely positive maps $\Psi_{j}: \cM \rightarrow \cM$ that witness the symmetric Haagerup property of $\cM$ and such that $\Psi_{j}(1) \leq 1$.  Let $\epsilon_j = \frac{1}{2^j}$ and let $\{ F_j \}_{j \in \mathbb{N}}$ be an increasing sequence of finite sets in the square of the Tomita algebra $\mathcal{T}_\varphi^2$ such that the union over $j \in \mathbb{N}$ of $ D_\varphi^{\frac{1}{4}} F_j D_\varphi^{\frac{1}{4}}$ is dense in $L^2(\cM, \varphi)$. Choose $n(j) \in \mathbb{N}$ such that for every $n' \geq n(j)$ we have,
\begin{equation}\label{Eqn=MaxLambda}
\max_{\lambda \geq 0} \lambda e^{-2n' (\lambda-1)^2} \leq 1 + \epsilon_j.
\end{equation}
If this property holds for $n$, then it automatically holds for $n' \geq n$. Let $C_n$ be a constant such that,
\begin{equation}\label{Eqn=CnConstant}
(e^{-n(\lambda -1)^2 } - 1)^2 \leq C_n (\lambda -1)^2,
\end{equation}
(an elementary check of the derivatives of these functions at $\lambda = 1$ implies that such a constant  indeed exists).

The spectrum $\Sigma(g_k^l)$ of $g_k^l \in \cM$ is real and the union $\cup_{k,l} \Sigma(g_k^l)$ is contained in a compact set since $g_k^l$ is uniformly bounded in $k$ and $l$. From \eqref{Eqn=ClaimNormAnalytic}, \eqref{Eqn=ClaimNormAnalyticII} and the fact that the invertible elements form an open set in $\cM$ it follows  that for each $\delta > 0$ we may find $l \in \mathbb{N}$ such that for every $k \in \mathbb{N}$ the imaginary parts of the spectra $\Sigma(  \sigma^\varphi_{-i/2}(g_k^l) )$ and  $\Sigma(  \sigma^\varphi_{-i/4}(g_k^l) )$ are contained in $[- \delta, \delta]$. The limits \eqref{Eqn=ClaimNormAnalytic} and \eqref{Eqn=ClaimNormAnalyticII} together with continuity of (the power series of) $F_n$ imply that
there exists an $l := l(j) \in \mathbb{N}$ such that for every $k \in \mathbb{N}$,
\begin{equation}\label{Eqn=AnalyticEstimate}
\begin{split}
\Vert \sigma_{-i/4}^\varphi( f_k^{n(j), l}  ) \Vert \leq 1 +  \epsilon_j,& \qquad \Vert \sigma_{-i/4}^\varphi( f_k^{n(j), l}  ) - f_k^{n(j), l}   \Vert \leq  \epsilon_j, \\
 \Vert \sigma_{-i/2}^\varphi( f_k^{n(j), l}  ) \Vert \leq 1 +  \epsilon_j,& \qquad \Vert \sigma_{-i/4}^{\varphi}( g_{k}^{l} ) - g_k^l \Vert \leq \frac{\epsilon_j}{\sqrt{C_{n(j)}}}.
\end{split}
\end{equation}
Since for every $l \in \mathbb{N}$, we have  $T_k^l \rightarrow 1$ strongly in $k$, we  can choose $k := k(j)$ such that for every $x \in F_j$,
\[
\begin{split}
\Vert (T_{k(j)}^{l(j)} -1 ) D_\varphi^{\frac{1}{4}}xD_\varphi^{\frac{1}{4}}\Vert_2 \leq \epsilon_j, \qquad
 \Vert (T_{k(j)}^{l(j)} - 1) D_\varphi^{\frac{1}{2}} \Vert_2 \leq & \frac{\epsilon_j}{\sqrt{C_{n(j)}}}.
\end{split}
\]

Now, set
\[
\Psi_j( \: \cdot \:) = \frac{1}{(1+  \epsilon_j)^2} f_{k(j)}^{n(j), l(j)} \Phi_{k(j)}^{l(j)}(\: \cdot \:) f_{k(j)}^{n(j), l(j)}.
\]
By \eqref{Eqn=MaxLambda} we have $\Psi_{j}(1) \leq \frac{1+\epsilon_j}{(1+ \epsilon_j)^2} \leq 1$. Also, for $x \in \cM^+$ the estimate \eqref{Eqn=AnalyticEstimate} yields the following estimate,  which uses in the first inequality  a standard application of Tomita-Takesaki theory (see \cite[Lemma 2.5]{HaaJunXu} for a proof),
\begin{equation} \label{Eqn=Weightreducing}
\varphi \circ \Psi_j(x) \leq  \frac{1}{(1+   \epsilon_j)^2}  \Vert \sigma^\varphi_{-i/2}( f_{k(j)}^{n(j), l(j)}) \Vert^2 \varphi \circ \Phi_{k(j)}^{l(j)}(x) \leq    \varphi (x).
\end{equation}
The KMS $L^2$-implementation of $\Psi_j$ is given by
\[
S_j:= \frac{1}{(1+  \epsilon_j)^2} \sigma_{-i/4}^\varphi( f_{k(j)}^{ n(j), l(j) }) J \sigma_{-i/4}^\varphi(f_{k(j)}^{n(j), l(j)}) J T_{k(j)}^{l(j)},
\]
which is compact and its norm is uniformly bounded in $j$ (this follows from $\Psi_j(1) \leq 1$ and the Kadison-Schwarz inequality). We shall now prove that $S_j \rightarrow 1$ strongly.   We  estimate,
\begin{equation}\label{Eqn=HugeEstimate}
\begin{split}
& \Vert \frac{1}{(1+  \epsilon_j)^2}  \sigma_{-i/4}^\varphi( f_{k(j)}^{n(j), l(j)}) J \sigma_{-i/4}^\varphi( f_{k(j)}^{n(j), l(j)} ) J T_{k(j)}^{l(j)} D_\varphi^{\frac{1}{4}}xD_\varphi^{\frac{1}{4}} -  D_\varphi^{\frac{1}{4}}xD_\varphi^{\frac{1}{4}} \Vert_2 \\
\leq & \Vert \frac{1}{(1+ \epsilon_j)^2} \sigma_{-i/4}^\varphi(f_{k(j)}^{n(j), l(j)}) J \sigma_{-i/4}^\varphi( f_{k(j)}^{n(j), l(j)} ) J T_{k(j)}^{l(j)}  D_\varphi^{\frac{1}{4}}xD_\varphi^{\frac{1}{4}} \\
 &\qquad-  \sigma_{-i/4}^\varphi(f_{k(j)}^{n(j), l(j)}) J \sigma_{-i/4}^\varphi( f_{k(j)}^{n(j), l(j)} ) J T_{k(j)}^{l(j)}  D_\varphi^{\frac{1}{4}}xD_\varphi^{\frac{1}{4}} \Vert_2 \\
& + \Vert  \sigma_{-i/4}^\varphi( f_{k(j)}^{n(j), l(j)}) J \sigma_{-i/4}^\varphi( f_{k(j)}^{n(j), l(j)} ) J T_{k(j)}^{l(j)} D_\varphi^{\frac{1}{4}}xD_\varphi^{\frac{1}{4}}  \\
&\qquad -  \sigma_{-i/4}^\varphi(f_{k(j)}^{n(j), l(j)}) J \sigma_{-i/4}^\varphi( f_k^{n(k), l(k)} ) J   D_\varphi^{\frac{1}{4}}xD_\varphi^{\frac{1}{4}} \Vert_2 \\
&  + \Vert  \sigma_{-i/4}^\varphi(f_{k(j)}^{n(j), l(j)}) J \sigma_{-i/4}^\varphi( f_{k(j)}^{n(j), l(j)} ) J   D_\varphi^{\frac{1}{4}}xD_\varphi^{\frac{1}{4}}  -  \sigma_{-i/4}^\varphi( f_{k(j)}^{n(j), l(j)})  D_\varphi^{\frac{1}{4}}xD_\varphi^{\frac{1}{4}} \Vert_2 \\
&  + \Vert  \sigma_{-i/4}^\varphi( f_{k(j)}^{n(j), l(j)}) D_\varphi^{\frac{1}{4}}xD_\varphi^{\frac{1}{4}}   -  D_\varphi^{\frac{1}{4}}xD_\varphi^{\frac{1}{4}} \Vert_2.
\end{split}
\end{equation}
The first summand on the right hand side converges to 0 as $j \rightarrow \infty$, since $\Vert T_{k(j)}^{l(j)} \Vert$ and $\|\sigma_{-i/4}^\varphi( f_{k(j)}^{n(j), l(j)} )\|$ are all bounded in $j$ and $\frac{1}{(1+ \epsilon_j)^2} \rightarrow 1$. Also, the second summand converges to 0 by our choice of $k(j)$. It remains to estimate the latter two summands.

By the assumption \eqref{Eqn=CnConstant} we have $(f_{k(j)}^{n(j), l(j)} - 1)^2 \leq C_{n(j)} (\Phi_{k(j)}^{l(j)}(1)  - 1)^2$. So,
\begin{equation}\label{Eqn=FirstEQN}
\begin{split}
& \Vert \sigma_{-i/4}^\varphi(f_{k(j)}^{n(j), l(j)} ) J \sigma_{-i/4}^{\varphi}(f_{k(j)}^{n(j),l(j)} ) J D_\varphi^{\frac{1}{4}} x D_\varphi^{\frac{1}{4}} - \sigma_{-i/4}^\varphi(f_{k(j)}^{n(j), l(j)}) D_\varphi^{\frac{1}{4}} x D_\varphi^{\frac{1}{4}} \Vert_2 \\ \leq & (1+ \epsilon_j) \Vert D_\varphi^{\frac{1}{4}} (x f_{k(j)}^{n(j), l(j)} - x) D_\varphi^{\frac{1}{4}} \Vert_ 2\\
\leq & (1+\epsilon_j) \Vert \sigma_{-i/4}^\varphi(x) \Vert \Vert (\sigma^\varphi_{-i/4}( f_{k(j)}^{n(j), l(j)})-1 ) D_\varphi^{\frac{1}{2}} \Vert_2 \\
\leq & (1+ \epsilon_j) \Vert \sigma_{-i/4}^\varphi(x) \Vert ( \Vert (f_{k(j)}^{n(j), l(j)} - 1) D_\varphi^{\frac{1}{2}}  \Vert_2 + \Vert ( \sigma_{-i/4}^\varphi(f_{k(j)}^{n(j), l(j)}) - f_{k(j)}^{n(j), l(j)}) D_\varphi^{\frac{1}{2}} \Vert_2 )\\
\leq & (1+ \epsilon_j) \Vert \sigma_{-i/4}^\varphi(x) \Vert ( \Vert (f_{k(j)}^{n(j), l(j)} -1) D_\varphi^{\frac{1}{2}}\Vert_2 + \epsilon_j \Vert D_\varphi^{\frac{1}{2}} \Vert_2 ).
\end{split}
\end{equation}
We have,
\begin{equation}\label{Eqn=SecondEQN}
\begin{split}
\Vert (f_{k(j)}^{n(j), l(j)} -1 ) D_\varphi^{\frac{1}{2}} \Vert_2^2 = & \varphi( (f_{k(j)}^{n(j), l(j)} -1)^2 )
\leq  C_{n(j)} \varphi( (\Phi_{k(j)}^{l(j)}(1) -1 )^2 )\\
=   C_{n(j)} \varphi( (g_{k(j)}^{l(j)} - 1  )^2 )
= & C_{n(j)} \Vert (g_{k(j)}^{l(j)} - 1  )D_\varphi^{\frac{1}{2}} \Vert_2^2.
\end{split}
\end{equation}
And in turn,
\begin{equation}\label{Eqn=ThirdEQN}
\begin{split}
& \sqrt{C_{n(j)}} \Vert (g_{k(j)}^{l(j)} - 1 ) D_\varphi^{\frac{1}{2}} \Vert_2 \\
\leq & \sqrt{C_{n(j)}} \left( \Vert ( \sigma_{-i/4}^\varphi(g_{k(j)}^{l(j)}) - 1)  D_\varphi^{\frac{1}{2}} \Vert_2 + \Vert (\sigma_{-i/4}^\varphi( g_{k(j)}^{l(j)} ) - g_{k(j)}^{l(j)}  ) D_\varphi^{\frac{1}{2}} \Vert_2 \right) \\
\leq & \sqrt{C_{n(j)}} \Vert D_\varphi^{\frac{1}{4}} (g_{k(j)}^{l(j)} - 1 ) D_\varphi^{\frac{1}{4}}\Vert_2 + \epsilon_j\Vert D_\varphi^{\frac{1}{2}} \Vert_2   \\
= & \sqrt{C_{n(j)}}   \Vert (T_{k(j)}^{l(j)} - 1 ) D_\varphi^{\frac{1}{2}}\Vert_2 + \epsilon_j \Vert D_\varphi^{\frac{1}{2}} \Vert_2.
\end{split}
\end{equation}
Combining \eqref{Eqn=FirstEQN}, \eqref{Eqn=SecondEQN} and \eqref{Eqn=ThirdEQN} we find that,
\begin{align*}
\Vert \sigma_{-i/4}^\varphi(f_{k(j)}^{n(j), l(j)} ) & J \sigma_{-i/4}^{\varphi}(f_{k(j)}^{n(j),l(j)} ) J D_\varphi^{\frac{1}{4}} x D_\varphi^{\frac{1}{4}} - \sigma_{-i/4}^\varphi(f_{k(j)}^{n(j), l(j)}) D_\varphi^{\frac{1}{4}} x D_\varphi^{\frac{1}{4}} \Vert_2 \\ &
\leq     (1+ \epsilon_j)^2 \Vert \sigma_{-i/4}^\varphi(x) \Vert (\epsilon_j+ 2 \epsilon_j \Vert D_\varphi^{\frac{1}{2}} \Vert_2).
\end{align*}

This shows that the third term in \eqref{Eqn=HugeEstimate} tends to 0 as $j \rightarrow \infty$. Lastly, we estimate the fourth summand in \eqref{Eqn=HugeEstimate}. We have,
\begin{align*}
&\Vert \sigma_{-i/4}^\varphi( f_{k(j)}^{n(j), l(j)})   D_\varphi^{\frac{1}{4}} x D_\varphi^{\frac{1}{4}}  - D_\varphi^{\frac{1}{4}} x D_\varphi^{\frac{1}{4}} \Vert_2 \\
 \leq &  \Vert f_{k(j)}^{n(j), l(j)} D_\varphi^{\frac{1}{4}} x D_\varphi^{\frac{1}{4}} - D_\varphi^{\frac{1}{4}} x D_\varphi^{\frac{1}{4}} \Vert_2 + \Vert ( f_{k(j)}^{n(j), l(j)} - \sigma_{-i/4}^{\varphi}(f_{k(j)}^{n(j), l(j)} ) ) D_\varphi^{\frac{1}{4}} x D_\varphi^{\frac{1}{4}} \Vert_2 \\
 \leq & \Vert (f_{k(j)}^{n(j), l(j)} - 1) J\sigma_{-i/4}^{\varphi}(x^\ast)J D_\varphi^{\frac{1}{2}} \Vert_2 + \epsilon_j \Vert D_\varphi^{\frac{1}{4}} x D_\varphi^{\frac{1}{4}} \Vert_ 2\\
 \leq  &   \Vert \sigma_{-i/4}^\varphi( x^\ast) \Vert \Vert (f_{k(j)}^{n(j), l(j)} -1 ) D_\varphi^{\frac{1}{2}} \Vert_2 + \epsilon_j \Vert D_\varphi^{\frac{1}{4}} x D_\varphi^{\frac{1}{4}} \Vert_2.
\end{align*}
Next we have,
\begin{align*}
\Vert ( f_{k(j)}^{n(j), l(j)} - 1  ) D_\varphi^{\frac{1}{2}} \Vert_2^2 & = \varphi(  ( f_{k(j) }^{n(j), l(j)} - 1  )^2 )
\leq  C_{n(j)} \varphi( (\Phi_{k(j)}^{l(j)}(1) - 1  )^2 )\\
& \leq  C_{n(j)} \Vert (g_{k(j)}^{l(j)} - 1) D_\varphi^{\frac{1}{2}} \Vert_2^2.
\end{align*}
Then,
\begin{align*}
\sqrt{C_{n(j)}} \Vert (g_{k(j)}^{l(j)} - 1 )D_\varphi^{\frac{1}{2}} \Vert_ 2 &\leq \sqrt{C_{n(j)}} \Vert ( \sigma_{-i/4}^\varphi(g_{k(j)}^{l(j)}) -1) D_\varphi^{\frac{1}{2}} \Vert_2 + \sqrt{C_{n(j)}} \Vert ( \sigma_{-i/4}^\varphi(g_{k(j)}^{l(j)} )- g_{k(j)}^{l(j)}  )D_\varphi^{\frac{1}{2}} \Vert_2 \\
& \leq  \sqrt{C_{n(j)} } \Vert (T_{k(j)}^{l(j)} - 1) D_\varphi^{\frac{1}{2}} \Vert_2 + \sqrt{C_{n(j)}} \Vert ( \sigma_{-i/4}^\varphi(g_{k(j)}^{l(j)}) - g_{k(j)}^{l(j)} ) D_\varphi^{\frac{1}{2}} \Vert_2 \\
& \leq  \sqrt{C_{n(j)}} \Vert (T_{k(j)}^{l(j)} - 1) D_\varphi^{\frac{1}{2}} \Vert_2 + \epsilon_j \leq 2 \epsilon_j.
\end{align*}
Combining the previous three estimates, this proves that the fourth summand of \eqref{Eqn=HugeEstimate} converges to 0.
In all, $(\Psi_j)_{j=1}^{\infty}$ forms a sequence of completely positive maps with $\Psi_j(1) \leq 1$ that witnesses the symmetric Haagerup property for $(\cM, \varphi)$.
\end{proof}

\begin{propn}\label{Prop=Corner}
Let $(\cM, \varphi)$ be a von Neumann algebra equipped with normal, semifinite, faithful weight $\varphi$. Let $\{ e_n \}_{n \in \mathbb{N}}$ be a sequence of projections in the centralizer of $\varphi$ strongly convergent to the identity. If $(\cM, \varphi)$ has the symmetric Haagerup property then for every $n \in \mathbb{N}$ $(e_n \cM e_n, e_n \varphi e_n)$ has the symmetric Haagerup property. Conversely,  if for every $n \in \mathbb{N}$ $(e_n \cM e_n, e_n \varphi e_n)$ has the symmetric Haagerup property with contractive approximating cp maps then $(\cM, \varphi)$ has  the symmetric Haagerup property with contractive approximating cp maps.
\end{propn}
\begin{proof}
Let $\Phi_k$ be  completely positive maps for $(\cM, \varphi)$ witnessing the symmetric Haagerup property with KMS $L^2$-implementations $T_k:=\Phi_k^{(2)}$.  Put $\Psi_{k,n}( \: \cdot \: ) = e_n \Phi_k(\: \cdot \:) e_n$. Then $\Psi_{k,n}$ is  completely positive with $(e_n \varphi e_n) \circ \Psi_{k,n} \leq e_n \varphi e_n$ (see \cite[Lemma 2.3]{CasSka}) and its KMS $L^2$-implementation $T_{k,n}$ is given by
\[
e_n D_\varphi^{\frac{1}{4}} x D_\varphi^{\frac{1}{4}}e_n \mapsto e_n D_\varphi^{\frac{1}{4}} e_n \Phi_k(x) e_n D_\varphi^{\frac{1}{4}} e_n = e_n D_\varphi^{\frac{1}{4}} \Phi_k(x) D^{\frac{1}{4}}_\varphi e_n = e_n J e_n J T_k(D_\varphi^{\frac{1}{4}}x D_\varphi^{\frac{1}{4}}).
\]
The latter mapping is clearly compact. Let $F \subseteq \mphi$ be a finite set. Let $n \in \mathbb{N}$ be such that for every $x \in F$ we have $\Vert e_n D_\varphi^{\frac{1}{4}} x D_\varphi^{\frac{1}{4}} e_n - D_\varphi^{\frac{1}{4}} x D_\varphi^{\frac{1}{4}} \Vert_2 \leq \epsilon$. Then choose $k \in \mathbb{N}$ such that $\Vert T_k (e_n D_\varphi^{\frac{1}{4}}x D_\varphi^{\frac{1}{4}} e_n) - e_n D^{\frac{1}{4}}_\varphi x D^{\frac{1}{4}}_\varphi e_n \Vert_2 \leq \epsilon$. Then the triangle inequality gives $\Vert T_{k,n} (D_\varphi^{\frac{1}{4}}x D_\varphi^{\frac{1}{4}}) - D_\varphi^{\frac{1}{4}} x D_\varphi^{\frac{1}{4}}   \Vert_2 \leq 2 \epsilon$. This proves that $T_{k,n} \rightarrow 1$ strongly in $k$ by a $3\epsilon$-estimate and the fact that $T_{k,n}$ is uniformly bounded in $k$ by construction.

\vspace{0.3cm}

Next we prove the converse. So suppose that every $(e_n \cM e_n, e_n \varphi e_n)$ has the symmetric Haagerup property. Let $\Phi_{k,n}$ be the corresponding completely positive maps which by assumption may be taken contractive. Set $\Psi_{k,n}( \: \cdot \:) = \Phi_{k,n}(e_n \: \cdot \: e_n)$. Then $\Psi_{k,n}$ is a contractive, completely positive map with $\varphi \circ \Psi_{k,n} \leq \varphi$  (see \cite[Lemma 2.3]{CasSka}) and its KMS $L^2$-implementation is given by
\[
S_{k,n}: D_\varphi^{\frac{1}{4}} x D_\varphi^{\frac{1}{4}} \mapsto D_\varphi^{\frac{1}{4}} \Psi_{k,n}(x) D_\varphi^{\frac{1}{4}} = T_{k,n}e_n J e_n J (D_\varphi^{\frac{1}{4}} x D_\varphi^{\frac{1}{4}}),
\]
which is compact. Let $F \subseteq \mphi$ be finite. Choose $n \in \mathbb{N}$ such that for every $x \in F$ we have $\Vert e_n D_\varphi^{\frac{1}{4}} x D_\varphi^{\frac{1}{4}} e_n - D_\varphi^{\frac{1}{4}} x D_\varphi^{\frac{1}{4}} \Vert_2 \leq \epsilon$. Next choose $k \in \mathbb{N}$ such that $\Vert T_{k,n} (e_n D_\varphi^{\frac{1}{4}} x D_\varphi^{\frac{1}{4}} e_n   ) - e_n D_\varphi^{\frac{1}{4}} x D_\varphi^{\frac{1}{4}} e_n \Vert_2 \leq \epsilon$. The triangle inequality then yields $\Vert S_{k,n} (D_\varphi^{\frac{1}{4}} x D_\varphi^{\frac{1}{4}} ) - D_\varphi^{\frac{1}{4}} x D_\varphi^{\frac{1}{4}}  \Vert_2 \leq 2 \epsilon$. The lemma follows then from a $3\epsilon$-argument and the fact that $\Psi_{k,n}$ being contractive implies through the Kadison-Schwarz inequality that $T_{k,n}$ is contractive for every $k,n \in \mathbb{N}$.
\end{proof}

\begin{propn}\label{Prop=RadonNikodym}
Let $\cM$ be a finite von Neumann algebra equipped with normal semifinite faithful trace $\tau$. Let $h \in \cM^+$ be a positive self-adjoint operator with trivial kernel (so its inverse exists as an unbounded operator) and let $\varphi(\: \cdot \:) =  \tau(h\: \cdot \: h)$ (formally, so the Connes cocycle derivative is determined by $(D\varphi/D\tau)_t = h^{2it}, t \in \mathbb{R}$). Then $(\cM, \varphi)$ has the symmetric Haagerup property if and only if $(\cM, \tau)$ has the symmetric Haagerup property. If either of these two pairs has the symmetric Haagerup property then the approximating cp maps may be taking contractive.
\end{propn}
\begin{proof}
Let $(\Phi_k)_{k=1}^{\infty}$ be completely positive maps witnessing the  symmetric Haagerup property of $(\cM, \varphi)$. Suppose first that $h$ is bounded and boundedly invertible. Then set $\Psi_k(\: \cdot \:) = h \Phi_k(h^{-1} \: \cdot \: h^{-1} ) h$. The maps $\Psi_k$ are completely positive, $\tau \circ \Psi_k \leq \tau$ and their KMS $L^2$-implementations are given by $h J h J T_k h^{-1} J h^{-1} J$, so are compact and converge to 1 strongly as $k$ tends to infinity.

In the case of a general $h$ consider $e_n = \chi_{(\frac{1}{n}, n)}(h)$. Then $(\cM, \tau)$ has the symmetric Haagerup property if and only if $(\cM, \tau)$ has the symmetric Haagerup property with contractive approximating cp maps (Lemma \ref{Lem=ExtraJolissaint}) if and only if for every $n \in \mathbb{N}$, $(e_n \cM e_n, e_n \tau e_n)$ has the symmetric Haagerup property with contractive approximating cp maps (Proposition \ref{Prop=Corner}) if and only if for every $n \in \mathbb{N}$, $(e_n \cM e_n, e_n \varphi e_n)$ has the symmetric Haagerup property with contractive approximating cp maps (first paragraph and Lemma \ref{Lem=MainResult}) if and only if $(\cM, \varphi)$ has the Haagerup property with contractive approximating cp maps (Proposition \ref{Prop=Corner} and its proof).
\end{proof}

\begin{lem}\label{Lem=Convergence}
Let $f_j(t) = \sqrt{\frac{j}{\pi}} e^{-\frac{1}{2} j t^2}$. Take $x \in \mathcal{T}_\varphi^2$. As $j \rightarrow \infty$ we have,
\[
\Vert D_\varphi^{\frac{1}{4}} \pi^{-1} \circ T ( \lambda(f_j) \pi(x) \lambda(f_j) ) D_\varphi^{\frac{1}{4}} - D_\varphi^{\frac{1}{4}}x D_\varphi^{\frac{1}{4}} \Vert_2 \rightarrow 0.
\]
\end{lem}
\begin{proof}
Note that $\lambda(f_j) \geq 0$ since the Fourier transform of $f_j$ is a positive function. We have, using \cite[Theorem 1.17]{TakII} and \cite[Lemma 3.5]{HaaKra},
\[
\begin{split}
& \Vert D_\varphi^{\frac{1}{4}} \pi^{-1} \circ T ( \lambda(f_j) \pi(x) \lambda(f_j) ) D_\varphi^{\frac{1}{4}} - D_\varphi^{\frac{1}{4}}x D_\varphi^{\frac{1}{4}} \Vert_2^2 \\
= &  \Vert \pi^{-1} \circ T ( \lambda(f_j) \pi(\sigma_{-i/4}^\varphi(x)) \lambda(f_j) ) D_\varphi^{\frac{1}{2}} -  \sigma_{-i/4}^\varphi(x) D_\varphi^{\frac{1}{2}} \Vert_2^2 \\
= & \int_{\mathbb{R}} \int_{\mathbb{R}} \vert f_j(t)\vert^2 \vert f_j(s)\vert^2  \varphi((\sigma_{-i/4-t}^\varphi(x)- \sigma_{-i/4}^\varphi(x))^\ast  (\sigma_{-i/4-s}^\varphi(x)- \sigma_{-i/4}^\varphi(x)) ) ds dt
\rightarrow  0.
\end{split}
\]
The identification of the limit follows from the fact that for $x, y \in \mathcal{T}_\varphi^2$ the mapping $\mathbb{R}^2 \ni (s,t) \mapsto \varphi(\sigma_s^\varphi(x) \sigma_t^\varphi(y))$ is continuous.
\end{proof}

At this point we recall some basic facts regarding crossed products. We use the same notational conventions as \cite{CasSka}. In particular, let $\cM$ be a von Neumann algebra with normal, semifinite, faithful weight $\varphi$. Let again $\sigma^\varphi: \mathbb{R} \rightarrow {\rm Aut}(\cM)$ be the modular automorphism group. Let $\cN = \cM \rtimes_{\sigma^\varphi} \mathbb{R}$ be the corresponding crossed product, also called the core of $\cM$. By $T: \cN^+ \rightarrow \cM^+_{{\rm ext}}$ we denote the canonical operator valued weight, so that $\tilde{\varphi} = \varphi \circ \pi^{-1} \circ T$ is the dual weight of $\varphi$. We let $\pi: \cM \rightarrow \cN$ be the natural embedding and for $f \in L^1(\mathbb{R})$, we let $\lambda(f)$ be the left regular representation of $f$ in the crossed product $\cN$. Finally $\mathcal{T}_\varphi$ denotes the Tomita algebra inside $\cM$. Recall that it consists of all $x \in \cM$ that are analytic for $\sigma^\varphi$ and such that $\sigma^\varphi_z(x) \in \nphi \cap \nphi^\ast$ with $z \in \mathbb{C}$. Let $\theta: \mathbb{R} \rightarrow {\rm Aut}(\cN)$ be the dual action. Then $\cN \rtimes_\theta \mathbb{R}$ is isomorphic to $\cM \otimes L^2(\mathbb{R})$ and the double dual weight $\tilde{\tilde{\varphi}}$ can be identified with $\varphi \otimes {\rm Tr}$. We shall regard $\cN$ as a subalgebra of $\cM \otimes L^2(\mathbb{R})$ (the embedding being canonical).  Let $S: (\cM \otimes L^2(\mathbb{R}) )^+ \rightarrow \cN^+_{{\rm ext}}$ denote the canonical operator valued weight, so that $\tilde{\tilde{\varphi}} = \tilde{\varphi} \circ  S$ is the double dual weight of $\varphi$. For more details we refer to Section 5 of \cite{CasSka}.

\begin{propn}\label{Prop=CrossedProduct}
Suppose that $(\cN, \tilde{\varphi})$ has the symmetric Haagerup property. Then $(\cM, \varphi)$ has the symmetric Haagerup property. Conversely, if $(\cM, \varphi)$ has the symmetric Haagerup property then so does $(\cN, \tilde{\varphi})$.
\end{propn}
\begin{proof}
Let $f_j(t) = \sqrt{\frac{j}{\pi}} e^{-\frac{1}{2} j t^2 }$, so that from the Fourier transform of
$f_j$, we see that the support of $\lambda(f_j)$ in $\cN$ equals 1 and that $\lambda(f_j) \geq 0$.
Furthermore $\lambda(f_j)$ is in the centralizer of $\tilde{\varphi}$ \cite{TakII}. Note also that
$\Vert f_j \Vert_{L^2 \mathbb{R}} = 1$. Define $\tilde{\varphi}_j( \: \cdot \:) = \tilde{\varphi}(
\lambda(f_j) \: \cdot \: \lambda(f_j))$. We see from Proposition \ref{Prop=RadonNikodym} together with the fact that $\cN$ is semifinite and Lemma \ref{Lem=ExtraJolissaint} that
$(\cN, \tilde{\varphi}_j)$ has the symmetric Haagerup property and that the approximating maps may
be chosen contractive. Let $\Phi_k^{(j)}$ be the completely positive contractive maps witnessing
this with $T_k^{(j)}$ their KMS $L^2$-implementations. Put
\[
\Psi_k^{(j)}( \: \cdot \:) = \pi^{-1} \circ T(\lambda(f_j) \Phi_k^{(j)} ( \pi( \: \cdot \: ) ) \lambda(f_j)  ).
\]
Then $\Psi_k^{(j)}$ is completely positive and $\varphi \circ \Psi_k^{(j)} \leq \varphi$ (see \cite[Lemma 5.2]{CasSka}). We need to show that the corresponding KMS $L^2$-implementations $S_k^{(j)}$ are compact and converge to the identity strongly for a suitable choice of a sequence in $j$ and $k$. Let
\[
U_j: L^2(\cM, \varphi) \rightarrow L^2(\cN, \tilde{\varphi}_j): D_\varphi^{\frac{1}{4}} x D_\varphi^{\frac{1}{4}} \mapsto D_{\tilde{\varphi}_j}^{\frac{1}{4}} \pi( x) D_{\tilde{\varphi}_j}^{\frac{1}{4}},
\]
and
\[
V_j: L^2(\cN, \tilde{\varphi}_j) \rightarrow L^2(\cM, \varphi): D_{\tilde{\varphi}_j}^{\frac{1}{4}} x D_{\tilde{\varphi}_j}^{\frac{1}{4}} \mapsto D_\varphi^{\frac{1}{4}} \pi^{-1} \circ T( \lambda(f_j) x \lambda(f_j) ) D_\varphi^{\frac{1}{4}}.
\]
It follows from Lemma \ref{Lem=L2Interpolation} and \cite[Lemma 3.5]{HaaKra} that these are well-defined contractive maps. Indeed, for example: let $x \in \cM^+$, then $\tilde{\varphi}_{j}(x) = \varphi \circ \pi^{-1} \circ T(\lambda(f_j) x \lambda(f_j)) = \varphi(x)$. Then  Lemma \ref{Lem=L2Interpolation} yields that $U_j$ is contractive. A similar argument holds for $V_j$.

 Now consider the composition $V_j T_k^{(j)} U_j$. We have,
\[
\begin{split}
& V_j T_k^{(j)} U_j (D_\varphi^{\frac{1}{4}} x D_\varphi^{\frac{1}{4}} ) \\
= &  V_j T_k^{(j)}   (D_{\tilde{\varphi}_j}^{\frac{1}{4}} \pi(x) D_{\tilde{\varphi}_j}^{\frac{1}{4}} ) \\
= & V_j D_{\tilde{\varphi}_j}^{\frac{1}{4}}  \Phi_k^{(j)}( \pi(x) ) D_{\tilde{\varphi}_j}^{\frac{1}{4}} \\
= & D_\varphi^{\frac{1}{4}} \pi^{-1} \circ T ( \lambda(f_j) \Phi_k^{(j)}(\pi(x)) \lambda(f_j) ) D_\varphi^{\frac{1}{4}}\\
= & S_k^{(j)} D_\varphi^{\frac{1}{4}}  x D_\varphi^{\frac{1}{4}}.
\end{split}
\]
This proves that $S_k^{(j)}$ is compact. Let $F \subseteq \mathcal{T}_\varphi^2$ be a finite set and let $\epsilon > 0$. First choose $j \in \mathbb{N}$ such that for every $x \in F$ we have,
\[
\Vert D_\varphi^{\frac{1}{4}} \pi^{-1} \circ T ( \lambda(f_j) \pi(x) \lambda(f_j) ) D_\varphi^{\frac{1}{4}} - D_\varphi^{\frac{1}{4}}x D_\varphi^{\frac{1}{4}} \Vert_2 \leq \epsilon.
\]
This is possible by Lemma \ref{Lem=Convergence}. Then choose $k \in \mathbb{N}$ such that for every $x \in F$ we have,
\[
\Vert T_k^{(j)} (D_{\tilde{\varphi}_j}^{\frac{1}{4}}\pi(x)D_{\tilde{\varphi}_j}^{\frac{1}{4}}) -  D_{\tilde{\varphi}_j}^{\frac{1}{4}}\pi(x)D_{\tilde{\varphi}_j}^{\frac{1}{4}} \Vert_2 \leq \epsilon.
\]
In all this gives,
\[
\begin{split}
&\Vert S_k^{(j)} (D_\varphi^{\frac{1}{4}} x D_\varphi^{\frac{1}{4}}) - (D_\varphi^{\frac{1}{4}} x D_\varphi^{\frac{1}{4}}) \Vert_2 \\
= & \Vert V_j T_k^{(j)} U_j (D_\varphi^{\frac{1}{4}} x D_\varphi^{\frac{1}{4}}) - V_j U_j (D_\varphi^{\frac{1}{4}} x D_\varphi^{\frac{1}{4}} )\Vert_2 + \Vert  V_j U_j (D_\varphi^{\frac{1}{4}} x D_\varphi^{\frac{1}{4}} )-   D_\varphi^{\frac{1}{4}} x D_\varphi^{\frac{1}{4}} \Vert_2 \leq 2 \epsilon.
\end{split}
\]
This proves the claim since the $S_k^{(j)}$'s are bounded uniformly in $k,j \in \mathbb{N}$.

\vspace{0.3cm}

We now prove the converse. The proof is essentially the same and we indicate the main differences. Firstly, since $(\cM, \varphi)$ has the symmetric Haagerup property so does $(\cM \otimes L^2(\mathbb{R}), \varphi \otimes {\rm Tr})$. Recall that $\varphi \otimes {\rm Tr}$ equals the double dual weight $\tilde{\tilde{\varphi}}$. Let $F_j = [-j, j]$ and set $f_j = \vert F_j \vert^{-1/2} \chi_{F_j}$. Consider $\nu(f_j) := 1 \otimes f_j \in \cM \otimes B(L^2(\mathbb{R}))$. Also set $e_j = \nu(\chi_{F_j}):= 1 \otimes \chi_{F_j}$ the support projection of $\nu(f_j)$. Set $\tilde{\tilde{\varphi}}_j( \: \cdot \: ) = \tilde{\varphi}( \nu(f_j) \: \cdot \: \nu(f_j))$. Since $\cM \otimes B(L^2(\mathbb{R}))$ has the symmetric Haagerup property, say with cp maps $\Phi_k$, so does $(e_j (\cM \otimes B(L^2(\mathbb{R})) ) e_j, \tilde{\tilde{\varphi}}_j)$ and we may take cp maps $\Phi_k^{(j)}( \: \cdot \: ) = e_j \Phi_k(\: \cdot \:) e_j$ (the proof is straightforward and uses that $e_j = \vert F_j\vert^{1/2} \nu(f_j)$ is in the centralizer of $\tilde{\tilde{\varphi}}$). Then put $\Psi_k^{(j)}( \: \cdot \:) = S(\nu(f_j)  \Phi_k^{(j)}( e_j \: \cdot \: e_j) \nu(f_j))$. As in the first part of the proof one can check that these maps witness the symmetric Haagerup property for $(\cN, \tilde{\varphi})$. At the places where Lemma \ref{Lem=Convergence} and \cite[Lemma 3.5]{HaaKra} are used one uses the computations made in \cite[Lemma 6.1 and Proposition 6.4]{CasSka} instead. The proof is similar to the considerations in \cite[Section 6]{CasSka}.
\end{proof}

\begin{lem}
Let $\cM$ be a semi-finite von Neumann algebra and $\varphi, \psi$ two normal, semifinite, faithful weights on $\cM$. $(\cM, \varphi)$ has the symmetric Haagerup property if and only if $(\cM, \psi)$ has the symmetric Haagerup property.
\end{lem}
\begin{proof}
This follows from Proposition \ref{Prop=RadonNikodym} since we may assume that $\psi$ is a trace and then $\varphi( \: \cdot \:) = \psi( h \: \cdot \: h)$ for some positive selfadjoint operator $h$ as in Proposition \ref{Prop=RadonNikodym}.
\end{proof}

Now we can conclude the following main result.

\begin{tw}
Let $\cM$ be a von Neumann algebra and $\varphi, \psi$ two normal, semifinite, faithful weights on $\cM$. $(\cM, \varphi)$ has the symmetric Haagerup property if and only if $(\cM, \psi)$ has the symmetric Haagerup property.
\end{tw}
\begin{proof}
The proof is now a mutatis mutandis copy of \cite[Theorem 5.6]{CasSka}. Note that it requires that $B(\Hil)$ has the symmetric Haagerup property, which follows from (the proof of) \cite[Proposition 3.4]{CasSka}. Also the proof of \cite[Lemma 3.5]{CasSka} remains valid.
\end{proof}

In the next lemma crossed product duality arguments are applied to show that Lemma \ref{Lem=MainResult} can be extended to arbitrary weights.

\begin{lem}\label{Lem=Contractive}
Let $(\cM, \varphi)$ be a pair of a von Neumann algebra with normal, semifinite, faithful weight $\varphi$. If $(\cM, \varphi)$ has the symmetric Haagerup property, then the completely positive maps $\{\Phi_k\}_{k \in \mathbb{N}}$ witnessing this may be chosen contractive.
\end{lem}
\begin{proof}
Assume that $(\cM,\varphi)$ has the symmetric Haagerup property. Let $\sigma^\varphi$ be the modular automorphism group of $\varphi$. Let $\cN = \cM \rtimes_{\sigma^\varphi} \mathbb{R}$ be the crossed product. We have $\cM \otimes B(L^2 \mathbb{R}) = \cN \rtimes_\theta \mathbb{R}$ where $\theta$ is the dual action on $\cN$. Since $\cM$ has the symmetric Haagerup property, so has $\cM \otimes B(L^2 \mathbb{R})$ (the proof is the same as \cite[Proposition 3.4 and Lemma 3.5]{CasSka}) and therefore $\cN$ has the symmetric Haagerup property by Proposition \ref{Prop=CrossedProduct}. If we can prove that the completely positive maps witnessing this for $\cN$ can be chosen contractive, then also $\cM$ has the symmetric Haagerup property with contractive completely positive maps by the proof of Proposition \ref{Prop=CrossedProduct}.

Let $h$ affiliated with $\cN$ be the generator of the left regular representation, i.e. $h^{it} = \lambda_t \in \cN$. Set $p_n = \chi_{[1/n, n]}(h)$. $\cN$ is well-known to be semifinite and carries a normal semifinite faithful trace $\tau$ such that $\tau(h^{\frac{1}{2}} \: \cdot \: h^{\frac{1}{2}}) = \tilde{\varphi}$. It follows from \cite[Lemma 5]{TerpI} that we have $\tau( p_n ) < \infty$. This means that the restrictions of both $\tau$ and $\tilde{\varphi}$ to $p_n \cN p_n$ are states and we have $\tau(p_n h^{\frac{1}{2}} \: \cdot \: h^{\frac{1}{2}} p_n) = \tilde{\varphi}( p_n \: \cdot \: p_n)$. From Proposition \ref{Prop=Corner} we know that $(p_n \cN p_n, \tilde{\varphi}(p_n \: \cdot \: p_n))$ has the symmetric Haagerup property and by Lemma \ref{Lem=MainResult}  we see that the completely positive maps of $\tilde{\varphi}(p_n \: \cdot \: p_n)$ may be taken contractive. From Proposition \ref{Prop=Corner}   it follows that $(\cN, \tilde{\varphi})$ has the Haagerup property with contractive maps. We conclude the lemma by the first paragraph.
\end{proof}

We now focus on the state case, beginning with a technical lemma, whose proof relies on ideas of \cite{TerpI}, \cite{TerpII}.

\begin{lem}\label{Lem=NCLP}
Let $\varphi$ and $\psi$ be normal states on a von Neumann algebra $\cM$. Suppose that $\psi \leq \varphi$. Then, there exists a $0 \leq c \leq 1$ such that $\psi = \varphi_c$, where $\varphi_c$ is defined by,
\[
\varphi_c(x) = \langle J x^\ast J \Lambda(c^{\frac{1}{2}}),\Lambda(c^{\frac{1}{2}}) \rangle.
\]
\end{lem}
\begin{proof}
Let $D_\psi$ and $D_\varphi$ be the elements in the Haagerup $L^1$-space associated with $\cM$ corresponding to $\psi$ and $\varphi$, see Proposition 15 of \cite{TerpI}. Since $\psi \leq \varphi$ we have $D_\psi \leq D_\varphi$. Then, $c := D_\varphi^{-\frac{1}{2}}\cdot D_\psi \cdot D_\varphi^{-\frac{1}{2}}$ is bounded and in fact $0 \leq c \leq 1$. This implies that $c$ is contained in the Haagerup $L^\infty$-space and we have $D_\psi = D_\varphi^{\frac{1}{2}} \cdot c \cdot D_\varphi^{\frac{1}{2}}$. But the latter operator corresponds exactly to $\varphi_c$ by \cite[Eqn. (38)]{TerpII}.
\end{proof}

We are ready to show that the approximating maps in the state case can be chosen to be Markov.

\begin{propn}\label{Prop=UnitalPreserving}
Let $(\cM, \varphi)$ be a pair of a von Neumann algebra $\cM$ and a normal faithful state
$\varphi$. Suppose that $(\cM, \varphi)$ has the symmetric Haagerup property witnessed by
contractive approximating maps  $\Phi_k$. Then there exist Markov maps $\Psi_k$ that witness the
symmetric Haagerup property of $(\cM, \varphi)$.
\end{propn}
\begin{proof}
Since $\varphi\circ \Phi_k \leq \varphi$ Lemma \ref{Lem=NCLP} yields that there exits a $0 \leq c_k \leq 1$ such that,
\[
\varphi \circ \Phi_k  = \varphi_{c_k}.
\]
If $\Phi_k(1) = 1$ then we set $\Psi_k = \Phi_k$. Note that in this case $\varphi(c_k) = \varphi(1)$ so that $\varphi(1 - c_k) = 0$ and since $1-c_k$ is positive this implies that $1-c_k = 0$. So $c_k = 1$ and $\varphi \circ \Phi_k = \varphi$. Else, define,
\[
a_k = \frac{1}{ \varphi(1-c_k)} (1 - \Phi_k(1)), \qquad b_k = 1-c_k.
\]
Put,
\[
\Psi_k(x) = \Phi_k(x) + a_k^{\frac{1}{2}} \varphi_{b_k}(x) a_k^{\frac{1}{2}}.
\]
Then,
\[
\begin{split}
\varphi \circ \Psi_k(x) = & \varphi \circ \Phi_k(x) + \varphi(a_k^{\frac{1}{2}} \varphi_{b_k}(x) a_k^{\frac{1}{2}} ) \\
= & \varphi_{c_k}(x) + \varphi_{b_k}(x) \frac{1}{\varphi(1-c_k)} \varphi(1 - \Phi_k(1)) \\
= & \varphi_{c_k}(x) + \varphi_{b_k}(x) = \varphi_{b_k + c_k}(x) = \varphi_1(x) = \varphi(x).
\end{split}
\]
Here the third last equality can be derived from \cite[Corollary 12]{TerpII}. Clearly, the KMS $L^2$-implementation of $\Psi_k$ is compact. Finally, let $S_k$ be the KMS $L^2$-implementation of $\Psi_k$ and $T_k$ of $\Phi_k$. Then, for $x \in \cM^+$ we see that,
\[
\Vert (S_k - 1) D_\varphi^{\frac{1}{4}}xD_\varphi^{\frac{1}{4}} \Vert_2 \leq \Vert (S_k - T_k) D_\varphi^{\frac{1}{4}}xD_\varphi^{\frac{1}{4}} \Vert_2 + \Vert (T_k - 1) D_\varphi^{\frac{1}{4}}xD_\varphi^{\frac{1}{4}} \Vert_2,
\]
and the second term goes to 0. So it remains to estimate,
\[
\begin{split}
& \Vert (S_k - T_k) D_\varphi^{\frac{1}{4}} x D_\varphi^{\frac{1}{4}} \Vert_2 = \Vert \varphi_{b_k}(x) D_\varphi^{\frac{1}{4}} a_k D_\varphi^{\frac{1}{4}} \Vert_2 = \varphi_{b_k} (x) \Vert D_\varphi^{\frac{1}{4}} a_k D_\varphi^{\frac{1}{4}} \Vert_2 \leq \Vert x \Vert \varphi(b_k) \Vert D_\varphi^{\frac{1}{4}} a_k D_{\varphi}^{\frac{1}{4}} \Vert_2 \\
= & \Vert x \Vert \Vert D_\varphi^{\frac{1}{4}} (1- \Phi_k(1)) D_\varphi^{\frac{1}{4}} \Vert_2 \\
= & \Vert x \Vert \Vert (1 - T_k) D_\varphi^{\frac{1}{2}}  \Vert_2 \rightarrow 0.
\end{split}
\]
This concludes the proof, since every element in $\cM$ can be written as the sum of four positive elements.
\end{proof}

The next theorem was one of the aims of this section. Note that for \emph{modular} Haagerup property (see the discussion after Theorem \ref{Thm=Equiv}) a corresponding fact was shown in \cite{OkaTom}.

\begin{tw}\label{Thm=MainTheorem}
If $\varphi$ is a normal state on $\cM$ and $(\cM, \varphi)$ has the symmetric Haagerup property, then the completely positive maps in Definition \ref{Dfn=HAPSymmetric}  may be chosen unital, i.e. $\Phi_k(1) = 1$ and state-preserving $\varphi \circ \Phi_k = \varphi$.
\end{tw}
\begin{proof}
The  result follows  by combining Lemma \ref{Lem=Contractive} with Proposition \ref{Prop=UnitalPreserving}.
\end{proof}

\section{Haagerup property for arbitrary von Neumann algebras revisited} \label{Sect=Equiv}

In this section we summarise the equivalence between various definitions of the Haagerup property for von Neumann algebras, collecting and building on the results of
\cite{CasSka}, \cite{OkaTom}, \cite{COST} and \cite{OkaTom2}, and focusing on the situation where we work with a faithful normal state. We make also some comments on the modular Haagerup property and the possible extensions to the weight case and finish by formulating the free product result.

\begin{tw} \label{Thm=Equiv}
Let $\mlg$ be a von Neumann algebra with a faithful normal state $\varphi$. Then the following conditions are equivalent (and in fact do not depend on the choice of the faithful normal state):
\begin{rlist}
\item there exists a sequence $\Phi_k: \cM \rightarrow \cM$ of normal, completely positive maps such that $\varphi \circ \Phi_k \leq \varphi$ and their GNS $L^2$-implementations are compact operators  converging to the identity of $L^2(\cM, \varphi)$ strongly (the  Haagerup property of \cite{CasSka});
\item there exists a sequence $\Phi_k: \cM \rightarrow \cM$ of normal, completely positive maps such that $\varphi \circ \Phi_k \leq \varphi$ and their KMS $L^2$-implementations are compact operators converging to the identity of $L^2(\cM, \varphi)$ strongly (the symmetric Haagerup property of Section \ref{Section=SymmetricHAP});
\item there exists a sequence $\Phi_k: \cM \rightarrow \cM$ of Markov maps such that  their GNS $L^2$-implementations are compact operators  converging to the identity of $L^2(\cM, \varphi)$ strongly (the   Haagerup property of \cite{DawFimSkaWhi});
\item there exists a sequence $\Phi_k: \cM \rightarrow \cM$ of KMS-symmetric Markov maps such that  their KMS $L^2$-implementations are compact operators converging to the identity of $L^2(\cM, \varphi)$ strongly (the  unital, selfadjoint symmetric Haagerup property of Section \ref{Section=SymmetricHAP});
\item there exists a sequence $\Phi_k: \cM \rightarrow \cM$ of Markov maps such that  their KMS $L^2$-implementations are compact operators converging to the identity of $L^2(\cM, \varphi)$ strongly (the  unital symmetric Haagerup property of Section \ref{Section=SymmetricHAP});
\item $\mlg$ has the standard form Haagerup property of \cite{OkaTom}.
\end{rlist}

\end{tw}
\begin{proof}
Implications (iii)$\Longrightarrow$(i) and (iv)$\Longrightarrow$(v)$\Longrightarrow$(ii) are trivial. The equivalence (i)$\Longleftrightarrow$(vi) was noticed in \cite{OkaTom} (see also \cite{COST}). The implication (ii)$\Longrightarrow$(v) is Theorem \ref{Thm=MainTheorem}.  The equivalence (v)$\Longleftrightarrow$(iii) was was shown  in \cite{OkaTom2} (the latter paper proves also that the equivalence remains true if we consider completely contractive, $\varphi$-reducing maps). The equivalence (i)$\Longleftrightarrow$(ii) follows from the arguments of Section \ref{Section=SymmetricHAP} -- specifically from the fact that $\mlg$ has the symmetric Haagerup property if and only if its core algebra has the symmetric Haagerup property, that the same fact is true for the Haagerup property of \cite{CasSka} and that the two versions of the Haagerup property are clearly equivalent for semifinite von Neumann algebras (as neither depends on the choice of the state).

Thus it remains to comment on the implication (v)$\Longrightarrow$(iv).  Let $(\Phi_n)_{n=1}^{\infty}$ be a sequence of approximating Markov maps on $\mlg$, as defined in (iv). Consider the sequence $(\Phi_n^{\dagger})_{n=1}^{\infty}$: it consists of Markov maps, their KMS implementations are compact (as $ (\Phi_n^{\dagger})^{(2)} = (\Phi_n^{(2)})^*$), and obviously $(\Phi_n^{\dagger})^{(2)}$ converge in the wo-topology to $I_{\ltwo}$ (as $(\Phi_n)^{(2)}$ did). Thus  $(\Phi_n^{\dagger})_{n=1}^{\infty}$ is a sequence of approximating Markov maps on $\mlg$, and so is $(\frac{\Phi_n+ \Phi_n^{\dagger}}{2})_{n=1}^{\infty}$. The latter are obviously KMS-symmetric.
\end{proof}

Naturally from now on we will simply say that a von Neumann algebra $\mlg$ with a separable predual \emph{has the Haagerup approximation property} if any/all conditions listed in the above theorem hold. In \cite{OkaTom} the authors consider also the case where one can find the approximating maps (in the sense of condition (vi) of Theorem \ref{Thm=Equiv}) commuting with the action of the modular group, and show that in that case they can be in addition assumed to be Markov (see Theorem 4.11 of that paper). If such maps exist, we will say that $\mlg$ has \emph{the modular Haagerup property}. 
In general the Haagerup property of Theorem \ref{Thm=Equiv} is not equivalent to the modular Haagerup property, see Corollary \ref{modularperiodic}.
Note however the following result, observed also by R.\,Tomatsu.

\begin{propn}\label{qgroupmodular}
Let $\QG$ be a discrete quantum group with the Haagerup property. Then the von Neumann algebra $L^{\infty}(\hQG)$ has the modular Haagerup property.
\end{propn}

\begin{proof}
 This is effectively a corollary of (the proofs of) Theorem 7.4 and Proposition 7.17 of \cite{DawFimSkaWhi}. We sketch the argument: as $\QG$ has the Haagerup property,
there exists a net of states $(\mu_i)_{i\in I} \in S(C^u(\hQG))$ such that the net $\Four(\mu_i) \in c_b(\QG)$ (see the notation in Section 6 of \cite{DawFimSkaWhi}) is in fact an approximate unit in $c_0(\QG)$. The proof of Proposition 7.17 in \cite{DawFimSkaWhi} shows that one can in addition assume that the states $\mu_i$ are invariant under the action of the universal scaling group $(\widehat{\tau}^u_t)_{t\in \br}$ (see \cite{kus}) -- effectively one uses the fact that $L^{\infty}(\br)$ admits an invariant mean and averages with respect to the action, the only non-trivial piece of the argument is showing that we still get an approximate unit in $c_0(\QG)$ -- this is however proved in the proposition mentioned above.
    Now Theorem 7.4 of \cite{DawFimSkaWhi} shows that each of the unital, normal, Haar state preserving completely positive maps $(L_i)_{i \in I}$ acting on $L^{\infty}(\hQG)$ defined by
\[    L_i(x)=(\id\otimes\mu_i)(\wW(x\otimes 1)\wW^*)\quad (x\in L^\infty(\hQG)) \]
has a compact implementation $T_i$ on $L^2(\hQG)$, and moreover $T_i$ tends strongly to $I_{L^2(\hQG)}$. It thus remains to verify that $\hat{\sigma}_t \circ L_i = L_i \circ \hat{\sigma}_t$ for $i \in I$, $t \in \br$. Note that abusing the notation we can view $\hat{\sigma}_t$ simply as an automorphism of $B(L^2(\hQG))$, so it makes sense to consider $(\hat{\sigma_t} \ot \hat{\tau}^u_t) (\wW)$. But
\begin{equation} (\hat{\sigma_t} \ot \hat{\tau}^u_t) (\wW) = \wW, \label{WW}\end{equation}
as follows from the formulas in Section 9 of \cite{kus} -- in particular note that Proposition 9.1 in the language of $\wW$ means that
\[ (\id \ot \hat{\tau}^u_t) (\wW) = (P^{-it} \ot \id)(\wW) (P^{it} \ot \id)\]
and the latter is equal to $(\tau_{-t} \ot \id)(\wW) = (\hat{\sigma}_{-t} \ot \id)(\wW)$ by Section 5.3 in \cite{kus2}.

The equation \eqref{WW} however means that
\begin{align*} \hat{\sigma}_t(L_i(x)) &=  (\id\otimes\mu_i) (\hat{\sigma}_t \ot \hat{\tau}^u_t) (\wW(x\otimes 1)\wW^*) \\&= (\id\otimes\mu_i \circ \hat{\tau}^u_t)(\wW(\hat{\sigma}_t(x)\otimes 1)\wW^*) =(\id\otimes\mu_i)(\wW(\hat{\sigma}_t(x)\otimes 1)\wW^*) \\&=
 L_i(\hat{\sigma}_t (x)), \end{align*}
which ends the proof.

\end{proof}

As mentioned before, the above fact was also observed by R.\,Tomatsu, who mentioned to us the following  proof of the above result (see also Corollary \ref{modularperiodic}):  the modular operator $\nabla_{\hat{\phi}}$ acting on $L^2(\hQG)$ is diagonalizable, in fact  the eigenvectors can be chosen among the entries of irreducible representations of $\hat{\QG}$. Thus Lemme 3.7.3 of \cite{ConnesThesis} implies that the action of the modular group $\{\hat{\sigma}_t:t\geq 0\}$ factorises through a compact group. This means that one can use the approximating maps constructed in Theorem 7.4 of \cite{DawFimSkaWhi} and average them with respect to the action of the modular group to obtain the modular Haagerup property.

We do not know if Theorem \ref{Thm=Equiv} (specifically the implication (i)$\Longrightarrow$(iii) can be extended to the case of faithful normal semifinite weights. In fact even the very special case of existence of Markov maps yielding suitable approximations for $(B(\ell^2), \Tr)$ remains open.

Finally we observe that the results of the last section allow us to improve on \cite[Theorem 4.12]{OkaTom}, establishing that the free product of von Neumann algebras with the modular Haagerup property has the (modular) Haagerup property by removing the modularity assumption.

\begin{cor}
Let $(\cM, \varphi)$ and $(\cN, \psi)$ be von Neumann algebras with normal faithful states. Suppose that $\cM$ and $\cN$ have the Haagerup property. Then also the von Neumann algebraic free product $\cM \star \cN$ with respect to the states $\varphi$ and $\psi$ has the Haagerup property.
\end{cor}
\begin{proof}
The proof is now the same as \cite[Proposition 3.9]{Boca} (or \cite[Theorem 4.12]{OkaTom})  since by Theorem \ref{Thm=MainTheorem} we may choose the completely positive maps witnessing the Haagerup property to be Markov.
\end{proof}

\section{Haagerup property via KMS symmetric Markov semigroups} \label{Sect=Semig}

The main result of this section is Theorem \ref{Thm:Semigroup}, describing the Haagerup property via Markov semigroups. We follow quite closely \cite{Sau} and \cite{JM}, at the same time clearly separating the purely Hilbert space-theoretic arguments from von Neumann algebraic considerations.

\begin{deft}
A \emph{Markov semigroup} $\{\Phi_t:t\geq 0\}$ on a von Neumann algebra $\mlg$ equipped with a faithful normal state $\varphi$ is a semi-group of Markov maps on $\mlg$ such that for all $x \in \mlg$ we have $\Phi_t(x) \stackrel{t\to 0^+} {\longrightarrow}\Phi_0(x) =x$ $\sigma$-weakly.
It is said to be \emph{KMS-symmetric} if each $\Phi_t$ is KMS symmetric, and \emph{immediately $L^2$-compact} if each of the maps $\Phi_t^{(2)}$ with $t>0$ is compact.
\end{deft}

The following extension of Lemma \ref{KMS2M} is straightforward.

\begin{lem} \label{KMS2Msemigroup}
If $\{T_t:t\geq 0\}$ is a $C_0$-semigroup consisting of symmetric $L^2$-Markov operators, then there exists a KMS-symmetric Markov semigroup $\{\Phi_t:t\geq 0\}$ on $\mlg$ such that $T_t=\Phi_t^{(2)}$ for each $t \geq 0$.
\end{lem}
\begin{proof}
Follows from Lemma \ref{KMS2M} and the considerations ending Section \ref{Sect=HaagerupLp}.
\end{proof}

As mentioned in the beginning of this section, we intend to separate the Hilbert space arguments from the von Neumann algebra setup, as is clearly possible via Lemma \ref{KMS2Msemigroup}. Thus we first formulate the key statement purely in the Hilbert space language (see Proposition \ref{HilProp}). We will need some straightforward lemmas from the semigroup theory. The first one is Proposition 9.1.2 of \cite{Arlect} (or Proposition 2.1 of \cite{Ouh}).

\begin{lem} \label{Ar1}
Let $X$ be a Banach space, let $-A$ be the generator of a contractive $C_0$-semigroup $\{P_t: t\geq 0\}$ on $X$ and let $C \subset X$ be a closed convex set. Then the following conditions are equivalent:
\begin{rlist}
\item $P_t (C) \subset C$ for each $t\geq 0$;
\item $\lambda(\lambda I_X + A)^{-1} (C) \subset C$ for each $\lambda >0$.
\end{rlist}
\end{lem}

The second result shows that in some cases the conditions above are very easy to check.

\begin{lem} \label{bddgen}
Let $\Hil$ be a  Hilbert space, $C\subset \Hil$ a closed convex set, $n \in \bn$ and let $T_1, \ldots, T_n \in B(\Hil)$ be selfadjoint contractions such that $T_i(C)\subset C$. Put $A= n I_{\Hil} - \sum_{i=1}^n T_i $, and let $A_t:=\exp(-tA)$, $t\geq 0$. Then we have $A_t (C) \subset C$. Moreover for each $\lambda >0$ we have $\lambda(\lambda I_X + A)^{-1} (C) \subset C$.
\end{lem}
\begin{proof}
It suffices to consider the case $n=1$ (putting $T=\frac{1}{n} \sum_{i=1}^n T_i$). Note also that by Lemma \ref{Ar1} the last statement in the lemma will follow once we prove the rest. Denote by $P$ the orthogonal projection onto $C$, and recall its action on $\xi \in \Hil$ is characterised by the following conditions: $P\xi \in C$ and  for all $\eta \in C$ we have $\textup{Re}\, (\la \eta - P \xi, \xi - P\xi \rangle) \leq 0$.

We will use the following quadratic form fact (Theorem 9.1.5 of \cite{Arlect} or Theorem 2.2 of \cite{Ouh} -- in the context of quantum Dirichlet forms the Hilbert space projection language was introduced in \cite{Cip}): it suffices to prove that for all $\xi \in \Hil$ we have
$\textup{Re}\, Q(P\xi, \xi-P\xi)\geq 0$, where $Q(\xi,\eta) := \la \xi, A \eta\ra$ for all $\xi, \eta \in \Hil$ (see the discussion in Section \ref{Dirichlet}).
We have however
\begin{align*} \textup{Re}\, Q(P\xi, \xi-P\xi) = \textup{Re}\, \la P\xi, (I-T) (\xi - P\xi) \ra = \textup{Re}\, \la P\xi - TP\xi,  \xi - P\xi \ra \geq 0,
\end{align*}
as $TP\xi \in C$.

\end{proof}

The following proposition is key for the main theorem of this section.

\begin{propn} \label{HilProp}
Let $\Hil$ be a separable Hilbert space, $(C_i)_{i \in \Ind}$ a family of closed convex sets in $\Hil$, and $(T_n)_{n=1}^{\infty}$ a family of operators on $\Hil$ satisfying the following conditions:
\begin{rlist}
\item each $T_n$ is a self-adjoint contraction;
\item for all $i \in \Ind$, $n \in \bn$ we have $T_n (C_i) \subset C_i$;
\item for each $\xi\in \Hil$ we have $\lim_{n\to \infty} T_n \xi = \xi$.
\end{rlist}
Then there exists a $C_0$-semigroup $\{S_t: t\geq 0\}$ of self-adjoint contractions leaving each of the sets $C_i$-invariant.
Moreover if each of the original $T_n$ is compact we can assume that $\{S_t: t\geq 0\}$ is immediately compact (recall this means that $S_t$ is compact for each $t>0$)
\end{propn}

\begin{proof}
The proof follows very closely these of \cite{JM} and \cite{Sau}, so we just indicate the main steps. We will also assume that the family $(C_i)_{i \in \Ind}$ consists of a single set $C$ (it will be clear that the same proof applies in general) and write $I$ for $I_{\Hil}$.

In the first step one shows (as in Theorem 1 of \cite{JM}) that without loss of generality we may assume that all maps $T_n$ mutually commute (replacing them by a family $(\wt{T}_n)_{n\in \bn}$ satisfying the same conditions plus the commutation requirement). To that end we first choose a dense subset $(\xi_l)_{l=1}^{\infty}$ in the unit ball of $\Hil$ and, by passing to a subsequence if necessary, assume that for all $n\in \bn$ and $\xi \in \Lin \{T_{j}^k (\xi_l): j,l=1,\ldots, n-1, k=0,1,\ldots, n^2\}$ we have
$ \|T_n \xi - \xi \|\leq 2^{-n} \|\xi\|.$ Then define (again for each $n \in \bn$) $\theta_n:= \frac{1}{n} (T_1+\cdots +T_n)$, $\Delta_n:= n(I-\theta_n)$.
Lemma \ref{bddgen} implies that for each $\lambda>0$ the element $R_{n,\lambda}:= \lambda(\lambda I+\Delta_n)^{-1}\in B(\Hil)$ is a self-adjoint contraction with  $R_{n,\lambda}(C)\subset C$.
Further we show by explicit estimates that for each $l\in \bn$ and $\lambda>0$ the sequence $(R_{n, \lambda} \xi_l)_{n=1}^{\infty}$ converges. This follows exactly as on pages 43--44 of \cite{JM}. The density argument allows us to define for each $\lambda >0$ a self-adjoint contraction $\rho_{\lambda}\in B(\Hil)$ as a strong limit:
\[ \rho_{\lambda}(\xi) = \lim_{n\to \infty} R_{n, \lambda^{-1}} \xi, \;\;\; \xi \in \Hil.\]
It is obvious that $\rho_{\lambda}(C) \subset C$.
Again explicit calculations using the resolvent formula (p.44 of \cite{JM}) show that for each $\xi \in \Hil$, $\mu, \nu >0$
\[ \lim_{\lambda \to 0^+} \rho_{\lambda} (\xi) = \xi, \;\;\; \lambda \rho_{\lambda} - \mu \rho_{\mu} = (\mu - \lambda) \rho_{\lambda} \rho_{\mu},\]
so putting $\wt{T_n} = \rho_{\frac{1}{n}}$ for each $n \in \bn$ we obtain the required mutually commuting maps satisfying the assumptions (i)-(iii).
We still need to argue that if the original maps $T_n$ were compact, the same will be true for $\wt{T_n}$. To that end, we define for all $n,m\in \bn, n<m$ the operator
$\Delta_{n,m}= \Delta_m - \Delta_n$, and follow the arguments on page 45 of \cite{JM} to obtain the following statements:
for each $n \in \bn$, $\lambda>0$ and $\xi \in \Hil$ the sequence $((I + \frac{\lambda}{n \lambda +1} \Delta_{n,m})^{-1}\xi)_{m=n+1}^{\infty}$ converges, so that we can define a new operator, say $\gamma_{\lambda,n}\in B(\Hil)$ as the strong limit of the sequence  $((I + \frac{\lambda}{n \lambda +1} \Delta_{n,m})^{-1})_{m=n+1}^{\infty}$. Finally putting ($n \in \bn$, $\lambda>0$)
\[ \psi_{n, \lambda}:=\theta_n (I + \lambda \Delta_n)^{-1} \gamma_{\lambda,n} + \frac{\lambda}{n \lambda +1} (\rho_{\lambda} - (I + \lambda \Delta_n)^{-1}) \theta_n \gamma_{\lambda, n}\]
we observe that if each $T_n$ is compact, so is $\theta_n$, and thus also, by the above formula $\psi_{n, \lambda}$. Finally, by the computations on pages 45-46 of \cite{JM} we obtain ($n \in \bn$, $\lambda>0$)
\[\|\rho_{\lambda} - \psi_{n, \lambda}\| \leq \frac{2}{n \lambda}.\]
This suffices to conclude that each $\rho_{\lambda}$ (and thus also each $\wt{T}_n$) is compact.

The second part of the proof follows the lines of Lemma 2 of \cite{JM}. We begin with a sequence $(T_n)_{n=1}^{\infty}$ of mutually commuting maps satisfying (i)-(iii) above. Fix a dense set $(\xi_l)_{l=1}^{\infty}$ in the unit ball of $\Hil$. Possibly passing to a subsequence, we can assume this time that for each $l\in \bn$ we have $\sum_{n=1}\|T_n(\xi_l) - \xi_l\| < \infty$. Define once again for $n, m \in \bn, n <m$, the maps $\theta_n:= \frac{1}{n} (T_1+\cdots T_n)$, $\Delta_n:= n(I-\theta_n)$, $\Delta_{n,m}= \Delta_m - \Delta_n$. Further for each $n \in \bn,$ $t\geq 0$ let $S_{t,n}= \exp(-t \Delta_n)$. Then each $\{S_{t,n}:t\geq 0\}$ is a family of self-adjoint contractions and Lemma \ref{bddgen} implies that $S_{t,n}(C) \subset C$. As $T_n$ are assumed to mutually commute, so do $\Delta_n$, so that we have for $m>n, t\geq 0$
the equality $S_{t,m} = \exp(-t\Delta_{n,m}) S_{t,n}$.  It is then easy to check that for each $l \in \bn$ the sequence $(S_{t,n} \xi_l)_{n=1}^{\infty}$ is convergent, so that further by a density argument we can define for $t\geq 0$ a self-adjoint contraction $S_t$ as a strong limit of $(S_{t,n})_{n=1}^{\infty}$. It is clear that $S_t (C)\subset C$. Further, as the arguments on page 40 of \cite{JM} show, $\{S_t:t\geq 0\}$ is a $C_0$-semigroup.

It remains thus to observe that if each of the original $T_n$ is compact, so are $S_t$ for $t>0$. This follows from the estimates obtained in \cite{JM} (pages 41-42): there exists $K>0$ such that for each $t>0$, $n \in \bn$
\[ \|(I - \theta_n)S_t \| \leq \|(I - \theta_n)S_{t,n}\| \leq \frac{K}{\sqrt{nt}},\]
so $S_t = \lim_{n\to \infty}  \theta_n S_t$, and the latter maps are clearly compact (as each $\theta_n$ is).

\end{proof}

For an operator $T$ on a Hilbert space $\Hil$ and $k \in \bn$ we denote by $T^{(k)}$ the natural matrix type lifting (which can be viewed also as tensoring) of $T$ to an operator on $M_{k} (\Hil) \approx \Hil \ot M_{k}$, where $M_{k}$ is viewed as the Hilbert space formed by Hilbert-Schmidt operators on $\bc^{k}$. The following generalization of the last proposition is now straightforward.

\begin{propn} \label{HilPropMatrix}
Let $\Hil$ be a separable Hilbert space, let $(k_i)_{i \in \Ind}$ be a collection of positive integers and for each $i\in \Ind$ let  $C_i \subset M_{k_i} (\Hil)$ be a closed convex set. Let $(T_n)_{n=1}^{\infty}$ be a family of operators on $\Hil$ satisfying the following conditions:
\begin{rlist}
\item each $T_n$ is a self-adjoint contraction;
\item for all $i \in \Ind$, $n \in \bn$ we have $T_n^{(k_i)} (C_i) \subset C_i$;
\item for each $\xi\in \Hil$ we have $\lim_{n\to \infty} T_n \xi = \xi$.
\end{rlist}
Then there exists a $C_0$-semigroup $\{S_t: t\geq 0\}$ of self-adjoint contractions leaving each of the sets $C_i$-invariant (by which we mean $S_t^{(k_i)} (C_i) \subset C_i$).
Moreover if each of the original $T_n$ is compact we can assume that $\{S_t: t\geq 0\}$ is immediately compact.
\end{propn}

\begin{proof}
It remains to observe that all the constructions (and norm estimates, for a fixed $k$) do not change under passing to  under tensoring/matrix lifting to $M_k \ot \Hil$.
\end{proof}

We are ready for the main result of this section. The implication (ii)$\Longrightarrow$(i) in the theorem below is immediate,  the point lies in the possibility of constructing the approximating semigroup out of the approximating sequence. Note that this is precisely the situation in the case of $L^{\infty}(\hQG)$, where $\QG$ is a discrete quantum group with the Haagerup property, as can be deduced from Theorem 7.18 of \cite{DawFimSkaWhi}  -- the approximating semigroup of maps can  in that case be built of multipliers associated to  states on $C_u(\hQG)$ forming a convolution semigroup of states.

\begin{tw} \label{Thm:Semigroup}
Consider a pair $(\mlg, \varphi)$ of a von Neumann algebra with a faithful normal state. The following are equivalent:
\begin{rlist}
\item $(\mlg, \varphi)$ has the Haagerup property;
\item there exists an immediately $L^2$-compact KMS-symmetric Markov semigroup $\{\Phi_t:t\geq 0\}$ on $\mlg$.
\end{rlist}
\end{tw}

\begin{proof}
As mentioned above, the implication (ii)$\Longrightarrow$(i) is clear. Assume then that $(\mlg, \varphi)$ has the  Haagerup property and let $(\Phi_n)_{n\in \bn}$ be the approximating KMS-symmetric Markov maps, which are $L^2$-compact, whose existence is guaranteed by Theorem \ref{Thm=Equiv} (iv). Apply Proposition \ref{HilPropMatrix} to the Hilbert space $\ltwoH$, the family of maps $(\Phi_n^{(2)})_{n\in \bn}$, and the closed convex sets $C_{-1}=\{\hsq\}$, $C_0=\{\xi \in \ltwoH: 0 \leq \xi \leq \hsq\}$,  $C_k = P^{(k)}$, where $k \in \bn$ and $P^{(k)}$ is the positive cone in $M_k \ot \Hil$. The proposition allows us to conclude the existence of an immediately compact $C_0$-semigroup built of KMS-symmetric $L^2$-Markov operators. Lemma \ref{KMS2Msemigroup} ends the proof.
\end{proof}

It is worth noting that KMS-symmetric Markov semigroups extend automatically to all Haagerup $L^p$-spaces (see for example \cite{GL1}). One can also define naturally a Haagerup approximation property  for $L^p(\cM, \varphi)$. This was done in \cite{OkaTom2}, where the authors showed also that $\cM$ has the Haagerup property if and only if so does any (equivalently, so do all) of the associated $L^p(\cM, \varphi)$. 

We end this section by stating the following corollary, which is essentially based on a remark of S.\,Neshveyev. Recall that a faithful normal state $\varphi$ on a von Neumann algebra is called \emph{almost periodic} if and only if the associated modular operator $\nabla_{\varphi}$ on $L^2(\mlg, \varphi)$ is diagonalizable.

\begin{cor} \label{modularperiodic}
Let $\mlg$ be a von Neumann algebra (with a separable predual) which has the Haagerup property and let $\varphi$ be a faithful normal state on $\mlg$. Then  $(\mlg, \varphi)$ has the modular Haagerup property if and only if $\varphi$ is almost periodic.    
\end{cor}

\begin{proof}
The backward implication, based on Lemme 3.7.3 in \cite{ConnesThesis} and described already after Proposition \ref{qgroupmodular}, is Theorem 4.14 in \cite{OkaTom}.
Assume then that  $(\mlg, \varphi)$ has the modular Haagerup property. It is easy to see that the construction in this Section preserve the commutation with the modular group, so a `modular' version of Theorem \ref{Thm:Semigroup} yields the existence of an immediately $L^2$-compact KMS-symmetric Markov semigroup $\{\Phi_t:t\geq 0\}$ such that each $\Phi_t$ commutes with the modular group. Passing to the $L^2$-picture we obtain an immediately compact semigroup $\{\Phi_t^{(2)}:t\geq 0\}$ which commutes with the unitary operators $\nabla_{\varphi}^{it}$ for each $t\in \br$. But then the generator of $\{\Phi_t^{(2)}:t\geq 0\}$ has finite dimensional eigenspaces $(V_n)_{n \in \bn}$ such that $L^2(\mlg, \phi)= \bigoplus_{n \in \bn} V_n$ (see Theorem \ref{compgen} in the following section) and each $\nabla_{\varphi}^{it}$ must preserve every $V_n$. This clearly implies that $\nabla_{\varphi}$ is diagonalizable.
\end{proof}

\section{Haagerup property via Dirichlet forms} \label{Dirichlet} \label{Sect=Dirichlet}

In this section we reformulate the statements obtained in Section \ref{Sect=Semig}. Begin by quoting the following standard/folklore result, whose detailed proof is available for example in Section 1.4 of \cite{Arlect} (we formulate a slightly different version, suited directly for our purposes).

\begin{tw} \label{compgen}
Let $\Hil$ be an (infinite-dimensional for simplicity of formulation) separable Hilbert space, and $\{T_t:t \geq 0\}$ a $C_0$-semigroup of self-adjoint contractions on $\Hil$ with the generator $-A$ (so that $A$ is a closed, self-adjoint positive operator on $\Hil$). Then the following conditions are equivalent:
\begin{rlist}
\item each $T_t$ with $t>0$ is compact;
\item there exists an orthonormal basis $(e_n)_{n\in \bn}$ in $\Hil$ and a non-decreasing sequence of non-negative numbers $(\lambda_n)_{n\in \bn}$ such that $\lim_{n\to \infty} \lambda_n = + \infty$ and
    \[ A e_n = \lambda_n e_n, \;\;\; n \in \bn.\]
\end{rlist}
\end{tw}
\begin{proof}
We sketch the proof: the implication (ii)$\Longrightarrow$(i) is immediate, as then we get an explicit spectral decomposition, $T_t = \sum_{n\in \bn} \exp(-t\lambda_n) |e_n \rangle \langle e_n |$.
The other implication is based on the following steps: first it is easy to see that (i) implies that the resolvent of $-A$ must be compact (Lemma 4.28 in \cite{Nagel}). Then we just use the standard spectral theorem for a self-adjoint compact operator to one of the operators $R_{\lambda} = \lambda (\lambda I + A)^{-1}$ in the resolvent (note that the resolvent consists of self-adjoint operators), and use the fact that it determines $A$ uniquely.
\end{proof}

For the general theory of closed quadratic forms we refer for example to \cite{Ouh}, and for quantum Dirichlet forms to \cite{GL1} and \cite{GL2}. We recall here the main results.

\begin{deft}
A non-negative closed densely defined quadratic form on a complex Hilbert space $\Hil$ is a map $Q:\Dom \, Q \to \br_+$ such that $\Dom\, Q$ is a dense subspace of $\Hil$, $Q$ is a quadratic form (i.e.\ there exists a sesquilinear form $\wt{Q}:\Dom\, Q \times \Dom \,Q \to \bc$ such that for all $\xi, \eta \in \Dom \, Q$ we have $\wt{Q}(\xi, \eta) = \overline{\wt{Q}( \eta, \xi)}$ and $Q(\xi) = \wt{Q}(\xi, \xi)$) and the space $\Dom \,Q$ equipped with the norm $\|\xi\|_{Q}:=\sqrt{\|\xi\|^2 + Q(\xi)}$ is complete.
\end{deft}

The following is a symmetric version of theorems in Section 1.5 of \cite{Ouh}.

\begin{tw} \label{formgen}
Let $\Hil$ be a Hilbert space.  There is a one-to-one correspondence between $C_0$-semigroups of self-adjoint contractions and non-negative closed densely defined quadratic forms on $\Hil$. Given a semigroup $\{P_t:t\geq 0\}$ the corresponding form $Q$ is given by the formula
\[  \Dom \, Q =\left\{ \xi \in \Hil: \sup_{t>0} \{\frac{1}{t} \la \xi, \xi - P_t \xi \ra\} < \infty \right\},\]
\[ Q(\xi) = \lim_{t\to 0^+} \frac{1}{t} \la \xi, \xi - P_t \xi \ra, \;\;\; \xi \in \Dom\, Q.\]
\end{tw}

The above theorem and Theorem \ref{compgen} yield immediately the following corollary.

\begin{cor} \label{Cor:formcompact}
Let $\Hil$ be an (infinite-dimensional for simplicity of formulation) separable Hilbert space, and let $Q$ be a non-negative closed densely defined quadratic form on $\Hil$, with the corresponding $C_0$-semigroup $\{T_t:t \geq 0\}$. Then the following conditions are equivalent:
\begin{rlist}
\item each $T_t$ with $t>0$ is compact;
\item there exists an orthonormal basis $(e_n)_{n\in \bn}$ in $\Hil$ and a non-decreasing sequence of non-negative numbers $(\lambda_n)_{n\in \bn}$ such that $\lim_{n\to \infty} \lambda_n = + \infty$, $\Dom \, Q = \{\xi \in \Hil: \sum_{n=1}^{\infty} \lambda_n |\la e_n, \xi  \ra|^2 < \infty\}$, and for $\xi \in \Dom \, Q$
    \[ Q(\xi) = \sum_{n=1}^{\infty} \lambda_n |\la e_n, \xi \ra|^2.\]
\end{rlist}
\end{cor}

Before we define (quantum) Dirichlet forms, we need to introduce some notations. For a selfadjoint $\xi \in \ltwoH_h$ we denote by $\xi_+$ the positive part of $\xi$ and further write $\xi_{\wedge}=\xi - (\xi - \hsq)_+$ (note that the operations $\xi \mapsto \xi_+$ and $\xi \mapsto \xi_{\wedge}$ correspond to taking orthogonal projections from $\ltwoH_h$ respectively onto the positive cone  and onto the set $\{\eta \in \ltwoH: \eta = \eta^*, \eta \leq \hsq\}$).

\begin{deft}
A non-negative closed densely defined quadratic form $Q$ on $\ltwoH$ is called a \emph{conservative, Dirichlet form}, if it satisfies the following conditions:
\begin{rlist}
\item $Q$ is real, i.e.\ $\Dom \, Q$ is $^*$-invariant and $Q(\xi) = Q(\xi^*)$ for $\xi \in \Dom \, Q$;
\item $\hsq \in \Dom \,Q $ and  $Q(\hsq) =0$;
\item for each self-adjoint $\xi \in \Dom \,Q $ we have $\xi_+, \xi_{\wedge} \in \Dom \,Q $ and $Q(\xi_+)\leq Q(\xi)$, $Q(\xi_{\wedge}) \leq Q(\xi)$.
\end{rlist}
It is called \emph{completely Dirichlet}, if the counterparts of conditions in (i)-(iii) are also satisfied for each $n \in \bn$ by the natural matrix lifting of $Q$, i.e. the quadratic forms $Q^{(n)}$ on $M_n \ot \ltwoH$ (with the domain $M_n (\Dom \, Q)$), defined by
\[ Q^{(n)} ([\xi_{ij}]_{i,j=1}^n ) = \sum_{i,j=1}^n Q(\xi_{ij}), \;\;\; \xi_{ij} \in \Dom \, Q, \, i,j=1,\ldots, n.\]
\end{deft}

The following result is contained in \cite{GL1} in the context of the Haagerup $L^2$-space. It was independently obtained in \cite{Cip} in the context of an arbitrary standard form; it can be viewed as a non-commutative version of the Beurling-Deny criteria and can be deduced from the interplay between semigroups, their generating forms, and closed convex sets (see Chapter 2 of \cite{Ouh}). Note that in fact \cite{Cip} shows also in Proposition 4.10 that the first property in condition (iii) above (the one related to $\xi_+$) follows automatically from the second (the one related to $\xi_{\wedge}$).

\begin{tw} \label{Thm:Dirichletsemig}
Let $(\mlg, \varphi)$ be a von Neumann algebra with a faithful normal state $\varphi$. There is a one-to-one correspondence between $KMS$-symmetric Markov semigroups on $\mlg$ and conservative completely Dirichlet forms on $\ltwoH$.
\end{tw}

The bijection above arises in the expected way: the quadratic form in question is the generating form (in the sense of Theorem \ref{formgen}) for the relevant semigroup of operators on $\ltwo$, connected to the semigroup on $\mlg$ via Lemma \ref{KMS2Msemigroup}.

We are ready to formulate the main result of this section. Note that once again quantum groups may be viewed as a guiding example:  convolution semigroups of states on the algebra of functions on a compact quantum group $\QG$ on one hand yield multipliers which can be used to prove the Haagerup property for the von Neumann algebra $L^{\infty}(\QG)$, and on the other hand lead to interesting quantum Dirichlet forms, studied recently in \cite{FabioUweAnna}.

\begin{tw} \label{Thm:DirichletHAP}
Consider a pair $(\mlg, \varphi)$. The following are equivalent:
\begin{rlist}
\item $(\mlg, \varphi)$ has the Haagerup property;
\item  $L^2(\mlg, \varphi)$  admits an orthonormal basis $(e_n)_{n\in \bn}$ and a non-decreasing sequence of non-negative numbers $(\lambda_n)_{n\in \bn}$ such that $\lim_{n\to \infty} \lambda_n = + \infty$  and the prescription
    \[ Q(\xi) = \sum_{n=1}^{\infty} \lambda_n |\langle e_n, \xi\rangle|^2, \;\;\; \xi \in \tu{Dom}\, Q,\]
where $\tu{Dom}\, Q=\{ \xi \in L^2(\mlg, \varphi):  \sum_{n=1}^{\infty} \lambda_n |\langle e_n, \xi\rangle|^2 < \infty\}$,   defines a conservative completely Dirichlet form.
    \end{rlist}
\end{tw}

\begin{proof}
It is a consequence of Theorem \ref{Thm:Semigroup}, Theorem \ref{Thm:Dirichletsemig} and Corollary \ref{Cor:formcompact}. The first two show that the Haagerup property for $\cM$ is equivalent to the existence of a conservative completely Dirichlet form $Q$ on $\ltwoH$ such that the $C_0$-semigroup of Hilbert space contractions associated to $Q$ is immediately $L^2$-compact and the last one interprets the immediate compactness of the semigroup directly in terms of $Q$.
\end{proof}

\subsection{Example}

We finish the article with an explicit description of a Dirichlet form with the properties of Theorem \ref{Thm:DirichletHAP} on a non-injective (finite) von Neumann algebra with the Haagerup approximation property. The example is based on the results of \cite{FabioUweAnna} and related to the observations in \cite{PierreRoland}.

Let $N\in \bn$, $N\geq 2$ and denote by $O_N^+$ the \emph{quantum orthogonal group of Wang} (see \cite{Wang}, and also \cite{Brannan} and references therein). The universal group $C^*$-algebra of $O_N^+$ is the universal unital $C^*$-algebra $\clg(O_N^+)$ generated by $N^2$ selfadjoint elements $\{u_{i,j}:i,j=1,\ldots, N\}$ such that the matrix $(u_{i,j})_{i,j=1,\ldots, N}$ is unitary. This algebra admits a unique bi-invariant (with respect to the natural coproduct) state $h$, which is tracial. The  von Neumann algebra associated to $O_N^+$, $L^{\infty}(O_N^+)$, is defined as the von Neumann completion of the image of $\clg(O_N^+)$ with respect to the GNS representation of $h$. Thus
$L^{\infty}(O_N^+)$ is a finite von Neumann algebra and $h$ induces a faithful tracial state on this algebra (we will denote it by the same letter).

It was proved by M.Brannan in \cite{Brannan} that $L^{\infty}(O_N^+)$ has the Haagerup approximation property -- in the language of \cite{DawFimSkaWhi} the dual of $O_N^+$ has the Haagerup property. Note that it follows from \cite{Teo} and \cite{BMT} that $L^{\infty}(O_N^+)$ is non-injective as soon as $N\geq3$. Explicit convolution semigroups of states on $\clg(O_N^+)$ were studied in Section 10 of \cite{FabioUweAnna}; in \cite{PierreRoland} it was observed explicitly that one of them has a `proper' generating functional and thus witnesses the Haagerup property of the dual of $O_N^+$. This is the basis of Proposition \ref{Prop:DirON} below. Before we formulate it we we need to introduce some more notations.

In \cite{Teo} T.\,Banica showed that the (equivalence classes of) irreducible representations of $O_N^+$ are indexed by non-negative integers; more explicitly he proved there (see also for example Section 10 of \cite{FabioUweAnna}) that if we denote for each $s\in \bn$ by $U_s$  the Chebyshev polynomial of the second kind, and put $n_s:=U_s(N)$, then there exists a linearly independent dense set $\{1, v_{i,j}^{(s)}:s \in \bn_0, i,j=1,\ldots, n_s\}$ in $\clg^*(O_N^+)$ such that the vectors $e_{0}:=\Omega_h$, $e_{i,j}^{s}:=
n_s^{-\frac{1}{2}}v_{i,j}^{(s)}\Omega_h, $ $s \in \bn_0, i,j=1,\ldots, n_s$, form an orthonormal basis in $L^2(O_N^+):= L^2(L^{\infty}(O_N^+), h)$.

\begin{propn} \label{Prop:DirON}
The following formula defines a completely conservative Dirichlet form on $L^2(O_N^+)$, satisfying the properties described in Theorem \ref{Thm:Dirichletsemig}:
\begin{equation} Q(\xi) = \sum_{s=1}^{\infty} \sum_{i,j=1}^{n_s} \frac{{U_s}'(N)}{U_s(N)}  |\langle e_{i,j}^s, \xi\rangle|^2, \;\;\; \xi \in \tu{Dom}\, Q, \label{DirON}\end{equation}
where
\[\tu{Dom}\, Q=\{ \xi \in L^2(O_N^+):   \sum_{s=1}^{\infty} \sum_{i,j=1}^{n_s} \frac{U_s(N)}{{U_s}'(N)}  |\langle e_{i,j}^s, \xi\rangle|^2 < \infty\}.\]

\end{propn}
\begin{proof}
Consider a linear, densely defined functional $L$ on $\clg(O_N^+)$ given by the formula
\[ L(1)=0, \;\;\; L(v_{i,j}^{(s)}) = -\frac{{U_s}'(N)}{U_s(N)}, \;\;\;s  \in \bn_0, i,j=1,\ldots, n_s.\]
Corollary 10.3 of \cite{FabioUweAnna} shows that it is a \emph{generating functional} on $O_N^+$ (it corresponds to $b=1$ and $\nu=0$ in that corollary). An elementary check shows that on its domain $L=L^{\dagger} \circ S$, where $S$ is the antipode of $O_N^+$, so that by Corollary 4.6 it is a $KMS$-symmetric functional. Further Theorem 7.1 of \cite{FabioUweAnna} implies that the formula \eqref{DirON} defines a Dirichlet form corresponding to the KMS-symmetric Markov semigroup of convolution operators/Schur multipliers associated with $L$ -- an apparent difference in comparison with the formulation there is related to the fact that we work with normalised eigenvectors.

The fact that the respective eigenvalues grow to infinity follows from explicit estimates: for $N=2$ we have
$\frac{U_s(2)}{{U_s}'(2)}= \frac{s(s+2)}{6}$ (Remark 10.4 of \cite{FabioUweAnna}), whereas for $N>2$ the sequence $(\frac{U_s(N)}{{U_s}'(N)} - \frac{s}{\sqrt{N^2-4}})_{s\in \bn}$ is bounded  (Lemma 4.4 of \cite{PierreRoland}).
\end{proof}

Note that as $h$ is tracial the choice of the embedding of $L^{\infty}(O_N^+)$ into $L^2(O_N^+)$ does not play any role in the above considerations.




\end{document}